\documentclass[11pt,reqno]{amsart}

 \usepackage{amsfonts,amsmath,amscd,amssymb,amsbsy,amsthm,amstext,amsopn,amsxtra}
 \usepackage{color,fullpage,mathrsfs}
 \usepackage{subfigure}
 \usepackage{verbatim} 
 \usepackage{stmaryrd}
 \usepackage{extarrows}
 \usepackage[all]{xy}
 \usepackage{bm}   
 \usepackage{hyperref}
 \usepackage{tikz}
 \usepackage{tikz-cd}
 \usepackage{enumerate}
 \usepackage{mathtools}
 \usepackage{cleveref}

\usepackage{cancel}
\usepackage{color}
\usepackage{tikz}

 \numberwithin{equation}{section}
\let\nc\newcommand
\let\renc\renewcommand

\theoremstyle{plain}
\newtheorem{thm}{Theorem}
\newtheorem{prop}[thm]{Proposition}
\newtheorem{cor}[thm]{Corollary}
\newtheorem{lem}[thm]{Lemma}

\theoremstyle{definition}
\newtheorem{defn}[thm]{Definition}
\newtheorem{example}[thm]{Example}

\theoremstyle{remark}

\newtheorem{remark}[thm]{Remark}
\newtheorem{rem}[thm]{Remark}
\numberwithin{thm}{section}

\newcommand{\Lmod}[1]{#1\text{-}{\mathsf{mod}}}
\newcommand{\grLmod}[1]{#1\text{-}{\mathsf{gmod}}}

\newcommand{\ot}{\otimes}

 \nc{\Qu}{\mathsf{Q}}
\nc{\op}{\mathrm{op}}

\newcommand{\idot}{{\:\raisebox{2pt}{\text{\circle*{1.5}}}}}

\DeclareMathOperator{\Ext}{\mathrm{Ext}}

\DeclareMathOperator{\im}{\mathrm{Im}}

\DeclareMathOperator{\Ker}{\mathrm{Ker}}

\DeclareMathOperator{\End}{\mathrm{End}}

\DeclareMathOperator{\rk}{\mathrm{rk}}

\DeclareMathOperator{\gr}{\mathrm{gr}}

\newcommand{\beq}{\begin{equation}\label}
\newcommand{\eeq}{\end{equation}}

\newcommand{\iso}{{\;\stackrel{_\sim}{\to}\;}}

\DeclareMathOperator{\Hom}{\mathrm{Hom}}

\DeclareMathOperator{\colim}{colim}

\nc{\Z}{\mathbb{Z}}
\newcommand{\N}{\mathbb{N}}

\newcommand{\C}{\mathbb{C}}

\nc{\rank}{\textrm{rank} \,}
\nc{\ds}{\dots}

\let\mc\mathcal
\let\mf\mathfrak

\nc{\mbf}{\mathbf}
\nc{\Res}{\mathsf{Res} \, }
\nc{\Ind}{\mathsf{Ind} \, }
\nc{\cont}{\textrm{cont}}
\renewcommand{\mod}{\textrm{mod}}
\nc{\msf}{\mathsf}

\nc{\minusone}{-1}
\nc{\minustwo}{-2}
\nc{\Mod}{\mathrm{Mod} \,}
\nc{\ms}{\mathscr}
\nc{\Frac}{\mathrm{Frac} \,}
\nc{\ra}{\rightarrow}
\nc{\hra}{\hookrightarrow}
\nc{\lab}{\label}
\renc{\O}{\mc{O}}
\nc{\Tan}{\mc{T}}
\nc{\ul}{\underline}
\nc{\s}{\mathfrak{S}}
\nc{\g}{\mf{g}}
\nc{\pa}{\partial}
\nc{\tit}{\textit}
\nc{\Maxspec}{\mathrm{Maxspec} \, }
\nc{\gldim}{\mathrm{gl.dim}}
\nc{\rkm}{\mathrm{rk} \, (\mf{m})}
\nc{\sm}{\mathrm{sm}}
\nc{\PD}{\mathbb{PD}}
\nc{\hilb}{\textrm{Hilb}}
\nc{\<}{\langle}
\renc{\>}{\rangle}
\nc{\T}{\mathbb{T}}
\nc{\X}{\mathbb{X}}
\nc{\F}{\mathbb{F}}
\nc{\id}{\msf{id}}
\nc{\A}{\mathbb{A}}
\nc{\Grat}{\mc{Grat}}
\nc{\Squo}[1]{\A^{(#1)}}
\nc{\twist}{\mathrm{twist}}
\nc{\Cd}{\mc{C}}
\nc{\Span}{\mathrm{Span}}
\nc{\Grass}{\mathrm{Gr}}
\nc{\Supp}{\mathrm{Supp}}
\nc{\Irr}{\mathrm{Irr}}
\newcommand{\ok}{\otimes_\rin}
\renc{\gr}{\mathsf{gr}}
\nc{\fin}{\mathrm{fin}}
\nc{\aff}{\mathrm{aff}}
\nc{\algD}{\mf{D}}
\nc{\hr}{\mf{h}_{\textrm{reg}}}
\nc{\D}{\mathscr{D}}
\nc{\PIdeg}{\mathrm{PI-degree}}
\nc{\ch}{\mathrm{ch}}
\nc{\ev}{\mathsf{ev}}
\nc{\Stab}{\mathrm{Stab}}
\nc{\Der}{\mathrm{Der}}
\nc{\rightsim}{\stackrel{\sim}{\longrightarrow}}
\nc{\HZ}{H_{\mbf{h},\Z}(\Z_m)}
\nc{\sing}{\mathrm{sing}}
\nc{\dd}{\mathscr{D}}
\nc{\bc}{\mathbf{c}}
\nc{\vc}{\underline{\mathbf{c}}}
\nc{\ba}{\mathbf{a}}
\nc{\reg}{\mathrm{reg}}
\nc{\Amp}{\mathrm{Amp}}
\nc{\Nef}{\mathrm{Nef}}
\nc{\SL}{\mathrm{SL}}
\nc{\Sp}{\mathrm{Sp}}
\nc{\Sym}{\mathrm{Sym}}
\nc{\Mov}{\mathrm{Mov}}
\nc{\Pic}{\mathrm{Pic}}
\nc{\rin}{k}
\nc{\Cs}{\C^{\times}}
\nc{\Nak}[3]{\mf{M}_{{#1}} ({#2},{#3}) }
\nc{\Naka}[2]{\mf{M}({#1},{#2}) }
\nc{\Mtheta}[1]{\mc{M}_{#1}}

\DeclareMathOperator{\Inj}{Inj}
\DeclareMathOperator{\Injn}{Inj-nil}
\nc{\bw}{\mathbf{w}}

\nc{\bn}{\mathbf{n}}
\nc{\CB}{\mathrm{CB}}
\nc{\GVect}{\Lambda}
\nc{\pZ}{\overline{Z}}
\nc{\Tang}{\mc{T}}
\nc{\K}{\mathbb{K}}
\renewcommand{\Im}{\mathrm{Im}\,}
\nc{\HOM}{\underline{\Hom}}
\nc{\EXT}{\underline{\Ext}}
\nc{\oR}{\otimes_R}

\newcommand{\mr}{\mathrm}
\newcommand{\coconnective}{connected graded }
\newcommand{\mb}{\mathbf}

\nc{\inj}{\mathrm{inj}}
\nc{\nil}{\mathrm{nil}}

\newcommand{\Fi}{\mb{F}}

\nc{\red}[1]{\textcolor{red}{#1}}

\definecolor{light}{gray}{.9}

\usepackage[backend=biber, bibstyle=mybibstyle, citestyle=numeric, date=year, sorting=nyt, giveninits=true]{biblatex}
\allowdisplaybreaks

\begin{document}

\title{Non-homogeneous Koszul duality in representation theory}

\author{Gwyn Bellamy}
\address{School of Mathematics and Statistics, University Place, Glasgow, G12 8QQ, Glasgow, UK.}
\email{gwyn.bellamy@glasgow.ac.uk}
\urladdr{https://www.gla.ac.uk/schools/mathematicsstatistics/staff/gwynbellamy/}

\author{Simone Castellan}
\address{School of Mathematics and Statistics, University Place, Glasgow, G12 8QQ, Glasgow, UK.}
\email{simone.castellan@glasgow.ac.uk}
\urladdr{https://sites.google.com/view/simonecastellan-math}

\author{Isambard Goodbody}
\address{School of Mathematics and Statistics, University Place, Glasgow, G12 8QQ, Glasgow, UK.}
\email{igoodbody.maths@gmail.com}
\urladdr{https://sites.google.com/view/isambard}

\begin{abstract}
Motivated by the representation theory of symplectic reflection algebras, deformed preprojective algebras, and graded Hecke algebras, we consider filtered algebras $U$ whose associated graded is Koszul. The Koszul dual of $U$, as defined by Positselski, is a curved dg-algebra. We establish an exact equivalence between the unbounded derived category of $U$ and an explicit quotient of the homotopy category of injective modules over the dual curved dg-algebra. This recovers a special case of a result of Positselski. In the case where $U$ has finite global dimension, the quotient is trivial and hence the unbounded derived category of $U$ is equivalent to the homotopy category of injective modules over the dual curved dg-algebra. 
\end{abstract}

\maketitle

\section{Introduction}

Koszul duality is an important concept in many areas of mathematics, ranging from representation theory and differential geometry to algebraic topology and commutative algebra. Its ubiquity has led to a number of abstract formulations, most notably due to Positselski \cite{positselskiiNonhomogeneousQuadraticDuality1993,positselskiTwoKindsDerived2011,positselskiRelativeNonhomogeneousKoszul2021} but also versions involving operads \cite{ginzburgKoszulDualityOperads1994,lodayOperadAssociativeAlgebras2010} and dg-categories \cite{holsteinCategoricalKoszulDuality2022}. In fact, Positselski has stated that ``it does not seem to admit a `maximal natural' generality." With this in mind, we take a different, example-led approach to Koszul duality. Specifically, we are motivated by examples from representation theory.

Koszul duality plays a central role in geometric representation theory, originally used to establish an equivalence \cite{beilinsonKoszulDualityPatterns1996a} between singular and parabolic category $\mathcal{O}$ for simple Lie algebras. It has since been generalized, for instance to Kac-Moody Lie algebras \cite{KacMoodyKoszul,BezYunKoszul} and to positive characteristic \cite{ModularKoszul}. Variations on this theme have been established for other important classes of algebras \cite{etingofKoszulityHilbertSeries2007} and Koszul duality plays a central role in symplectic duality for conic symplectic singularities \cite{BLPWAst}. 

In its simplest form, Koszul duality is a derived equivalence between certain categories of graded modules over a Koszul algebra $A$ and its dual $A^!$. A priori, it uses the grading on the algebras in an essential way. However, in seminal works by Positselski \cite{positselskiiNonhomogeneousQuadraticDuality1993,positselskiTwoKindsDerived2011,positselskiRelativeNonhomogeneousKoszul2021}, Koszul duality has been extended to the filtered setting. More precisely, one can substitute the Koszul algebra $A$ by a filtered deformation $U$. The fact that we move from graded to filtered algebras on one side of the duality translates (surprisingly) to considering curved dg-(co)algebras. These are dg-(co)algebras whose differential squares to a (in general) non-zero element $c$, called the curvature. This non-homogeneous Koszul duality can be applied to many more examples appearing in representation theory. The cost, however, is that the curvature introduces a  number of technical difficulties. 

First of all, non-zero curvature means that one cannot work with the ordinary notion of a derived category. To solve this problem,  Positselski introduces the coderived category and produces an equivalence \cite[Example~6.11, Theorem~6.12]{positselskiDifferentialGradedKoszul2023} between the derived category of a non-homogeneous Koszul algebra $U$ and the coderived category of its Koszul dual curved dg-coalgebra $A^?$: 
\begin{equation}\label{eq:Positequiv}
    D(\Lmod{U}) \iso D^{\mr{co}}(A^?\textrm{-}\mathsf{comod}).
\end{equation}

The equivalence \eqref{eq:Positequiv} is actually a special case of the general theory developed by Positselski; one does not need to make any kind of Koszulity assumption on $U$, but then one loses the description of $A^?$ as the quadratic dual. We encourage the interested reader to see his original papers to get the full picture. 

In many cases of interest to representation theorists, the situation can be simplified by avoiding reference to coalgebras. The coalgebra $A^?$ is replaced with a curved dg-algebra $A^!$ and the category of comodules with the category of \emph{locally nilpotent modules} (which are still called comodules in \cite{positselskiRelativeNonhomogeneousKoszul2021}). Then Positselski establishes\footnote{His result is more general, but specializes to the claimed equivalence in the case where the coefficient ring is semisimple.} in \cite[Corollary~6.18]{positselskiRelativeNonhomogeneousKoszul2021} an equivalence 
\begin{equation}\label{eq:Posit2}
       D(\Lmod{U}) \iso D^{\mr{co}}(A^!\textrm{-}\mathsf{mod})_\nil.
\end{equation}

The coderived category $D^{\mr{co}}(\Lmod{A^!})_\nil$ can be identified with the homotopy category $K(\Injn A^!)$ of locally nilpotent curved modules whose underlying graded $A^!$-module is injective \emph{in the category of locally nilpotent modules} \cite[Theorem~8.17]{positselskiRelativeNonhomogeneousKoszul2021}. Assuming $A^!$ is graded left Noetherian, we show that all such modules are injective in the category of graded $A^!$-modules and thus $D^{\mr{co}}(\Lmod{A^!})_\nil\simeq K(\Inj A^!)_{\mathrm{nil}}$. See Section~\ref{sec:locallynilp} for details.

As we explain in the next subsection, without any Noetherianity condition on $U$ or $A^!$, we define a certain explicit thick subcategory $\mathcal{N}$ of $K(\Inj A^!)$ and show that $D(U)$ is equivalent to $K(\Inj A^!)/\mc{N}$. Even when $A^!$ is graded left Noetherian, one cannot expect an equivalence between $D(U)$ and $K(\Inj A^!)$; we must take a proper quotient still. However, if we assume that $A^!$ is bounded or, equivalently, that $A$ has finite global dimension, we recover an equivalence $D(U)\simeq K(\Inj A^!)$. In particular, if we assume $U=A$ is a homogeneous Koszul algebra then we get an equivalence $D(A)\simeq K(\Inj A^!)/\mc N$. To the best of our knowledge, this result is new since classical Koszul duality only deals with subcategories of the derived categories $D(\grLmod{A})$ and $D(A^!)$ consisting of objects satisfying certain bounds on their grading; see \cite{beilinsonKoszulDualityPatterns1996a}.

The present paper has three goals. First, we give a direct, elementary proof of the equivalence $D(U)\simeq K(\Inj A^!)/\mc{N}$ and give various refinements under additional hypotheses. Secondly, we explain certain immediate applications of the equivalence. Finally, and perhaps most importantly, we explain that several classes of algebras that are of interest in geometric representation theory fit naturally into this framework. In particular, these include symplectic reflection algebras, deformed preprojective algebras, and graded Hecke algebras. We expect that it is a fruitful endeavour to translate representation-theoretic problems for these algebras into problems concerning curved dg-modules for $A^!$.  

\subsection{The equivalence}

We begin with a non-homogeneous quadratic algebra $U$ and assume that the associated graded $A:=\gr(U)$ is a Koszul algebra. The Koszul dual $A^!$ can be given the structure of a curved dg-algebra $(A^!,d,c)$. Consider the category $C(A^!)$ of curved dg-modules over $(A^!,d,c)$ and the category $C(U)$ of complexes of left $U$-modules. The bimodule $T:=U\ok A^!$ has a (curved) differential and allows one to define functors $F:=T\otimes_{A^!}-$ and $G:=\HOM_U(T,-)$. They give rise to a tensor-hom adjunction
\[
\begin{tikzcd}
C(U) \ar[r, shift left=1.0ex, "G"] & C(A^!) \ar[l, shift left=1.0ex, "F"],
\end{tikzcd}
\]
which descends to the homotopy categories
\[
\begin{tikzcd}
K(U) \ar[r, shift left=1.0ex, "G"] & K(A^!) \ar[l, shift left=1.0ex, "F"].
\end{tikzcd}
\]
For $M\in C(A^!)$, we define $S(M) = \left\{ m \in M \,  | \,  am = 0\  \forall a\in (A^!)^+ \right\}$. Let $K(\Inj A^!)$ be the subcategory of $K(A^!)$ consisting of curved dg-modules whose underlying graded module is injective, and consider the thick triangulated subcategory $\mathcal{N}:=\{I\in K(\Inj A^!) \, | \, S(I)\ \text{is acyclic} \}$. We show:

\begin{thm}\label{thm:mainequivalence}
    The functors $G,F$ induce inverse equivalences 
    \[
\begin{tikzcd}
D(U) \ar[r, shift left=1.0ex, "G"] & K(\Inj A^!)/\mathcal{N} \ar[l, shift left=1.0ex, "F"].
\end{tikzcd}
\]
\end{thm}

\begin{thm}\label{thm:main2}
Restricting to bounded above complexes, the category $\mathcal{N}$ is trivial and the equivalences of Theorem~\ref{thm:mainequivalence} restrict to
\[
\begin{tikzcd}
D^-(U) \ar[r, shift left=1.0ex, "G"] & K^-(\Inj A^!) \ar[l, shift left=1.0ex, "F"].
\end{tikzcd}
\]
Moreover, if $A$ has finite global dimension, then $\mathcal{N}$ is trivial even in unbounded complexes and we have equivalences 
\begin{equation}\label{eq:mainfgldim}
 \begin{tikzcd}
D(U) \ar[r, shift left=1.0ex, "G"] & K(\Inj A^!) \ar[l, shift left=1.0ex, "F"]
\end{tikzcd}   
\end{equation}
of the unbounded derived category. 
\end{thm}

In particular, the equivalence \eqref{eq:mainfgldim} holds for symplectic reflection algebras, deformed preprojective algebras, and graded Hecke algebras. 

If we assume that $A^!$ is left graded Noetherian, Theorem \ref{thm:mainequivalence} can be refined. Note that we do not place any Noetherianity condition on $U$ or $A$. 

\begin{thm}\label{thm:main3}
    If $A^!$ is graded left Noetherian, the intersection $\mc N\cap K(\Inj A^!)_\nil$ is zero. Every $G(M)\in K(\Inj A^!)/\mc N$ is canonically isomorphic to an object $\mb G(M)$ in $K(\Inj A^!)_\nil$ and there is an equivalence 
    \[
\begin{tikzcd}
D(U) \ar[r, shift left=1.0ex, "\mb G"] & K(\Inj A^!)_\nil  \ar[l, shift left=1.0ex, "F"].
\end{tikzcd}
\]
\end{thm}

All equivalences are proved using the following standard result for triangulated categories \cite[Proposition~3.2.9]{KrauseBook}: a conservative, triangulated functor $F:\mathcal{C}_1\rightarrow\mathcal{C}_2$ with a fully-faithful adjoint is an equivalence. In our case, $\mc{C}_1$ and $\mc{C}_2$ will be $K(\Inj A^!)/\mc N$ and $D(U)$, respectively.
Hence, the proof splits into two parts: proving the full-faithfulness of $G$ and proving that $\Ker F=\mathcal{N}$. 

Our approach is inspired by Fl{\o}ystad, which in \cite{floystadKoszulDualityEquivalences2005} deals with the special case of zero-curvature. In this case, we can still talk about acyclic dg-modules. Instead of our $\mc N$, Fl{\o}ystad considers two subcategories $\mc N_1\subset K(U)$ and $\mc N_2\subset K(A^!)$ of objects that are sent to acyclic objects by the functors $G$ and $F$, respectively. He then proves an equivalence $K(U)/\mc N_1\simeq K(A^!)/\mc N_2$. These categories are ``in between" the homotopy categories and the derived categories of $U$ and $A^!$.  

The first part of our proof is an easy generalization of his results, while the second part required a different approach. We reduced it to a spectral sequence argument, just as in the case of graded Koszul duality \cite{beilinsonKoszulDualityPatterns1996a}. Moreover, while \cite{floystadKoszulDualityEquivalences2005} considers only algebras over fields, we prove our results for algebras over semisimple rings. This is the setting in which Koszul duality was originally introduced in \cite{beilinsonKoszulDualityPatterns1996a}. Working over a semisimple ring creates technical difficulties, but allows us to consider interesting examples in representation theory, such as symplectic reflection algebras and deformed preprojective algebras.

\subsection{Applications}

In Section~\ref{sec:applications} we highlight a number of immediate applications of the equivalence of Theorem~\ref{thm:mainequivalence}. Let $(\Lambda,d,c)$ be a connected graded curved dg-algebra, whose underlying algebra is Koszul and graded left Noetherian. The standard $t$-structure on $D(U)$ gives rise to a $t$-structure on $K(\Inj \Lambda)_{\mathrm{nil}}$. The definition of this $t$-structure, in terms of the functor $S$, makes sense without the Koszul assumption on $\Lambda$, but we do not know if it is actually a $t$-structure without this assumption. If we assume further that $\Lambda$ is finite-dimensional and that its Koszul dual is Noetherian we show that its $K$-theory reduces to that of the base field. We also show, using Bousefield localization, that the inclusion $K(\Inj \Lambda)_{\mathrm{nil}} \hookrightarrow K(\Inj \Lambda)$ admits a left adjoint, the same being true for $K(\Inj \Lambda) \hookrightarrow K(\Lambda)$ when $\Lambda$ is finite-dimensional. 

Now assume that $U$ is a non-homogeneous Koszul algebra. The Koszul complex provides an explicit projective resolution of $U$ as a $U$-bimodule and hence a projective resolution for any (left or right) $U$-module. In the case of the deformed preprojective algebra associated to a finite connected non-Dynkin quiver, this is the resolution constructed by Crawley-Boevey \cite{crawley-boeveyDeformedPreprojectiveAlgebras2022}. This resolution gives a concise expression for the Hochschild cohomology of $U$ in terms of the Koszul dual $A^!$, a result first due to Negron \cite{negronCupProductHochschild2017}. We expect that the equivalence can be used to define shift, induction, and restriction functors for symplectic reflection algebras and deformed preprojective algebras.

\subsection{Examples} 
The general idea for non-homogeneous Koszul duality is that the linear part of the non-homogeneous quadratic relations gives rise to the differential on the dual side, while the scalar part induces the curvature. 

The standard example of Koszul duality is the symmetric algebra of a vector space $A=\Sym(V)$. The dual is the exterior algebra $A^!=\Lambda(V^*)$. We can see it as a special case of a curved dg-algebra with zero differential and zero curvature. If $V$ has a symplectic form $\omega$, we can deform $S(V)$ to obtain the Weyl algebra, with relations $\{u\ot v-v\ot u-\omega(u,v) \mid u,v\in V\}$. Thus we get a non-zero curvature $-\omega\in\Lambda^2(V^*)$ and a zero differential, since there is no linear part in the relations. If $V=\g$ is a Lie algebra, we can consider the enveloping algebra $U(\g)$ with defining relations $\{a\ot b-b\ot a-[a,b] \mid a,b\in\g\}$. Here we get a non-zero differential (obtained by dualizing the Lie bracket) on $\Lambda(\g^*)$, and zero curvature. In other words, the Koszul dual of $U(\g)$ is the standard cohomological complex of the Lie algebra $\mf g$. More generally, the deformations of $U(\g)$ introduced by Sridharan give examples where the curvature on $\Lambda(\g^*)$ is also non-zero. 

More interesting examples come from considering algebras over semisimple rings instead of fields. Let $\Gamma\subset \mathrm{Sp}(V)$ be a symplectic reflection group. The smash product $A=\Sym V \rtimes \Gamma$ is a Koszul algebra over the group algebra $\Bbbk\Gamma$. The symplectic reflection algebra $H_{t,c}(\Gamma)$ is a filtered deformation of $A$, depending on the parameters $(t,c)$. The relations only have a scalar part, so  $A^! = \wedge^{\idot} V^* \rtimes \Gamma$ has a trivial differential, but the pair $(t,c)$ defines a curvature element $c \in \wedge^{2} V^* \rtimes \Gamma$. 

Similarly, the preprojective algebra $A=\Pi(\Qu)$ of a quiver $\Qu$ can be seen as a quadratic algebra over the semisimple ring $k=\bigoplus_i\Bbbk e_i$, where $e_i$ are the trivial paths. It is Koszul if $\Qu$ is not of ADE type. It has a well-known filtered deformation, the deformed preprojective algebra $\Pi^\lambda(\Qu)$. Again, the relations only have a scalar quantum correction, so the Koszul dual $A^!$ is a curved dg-algebra with zero differential. As a vector space, $A^!$ is isomorphic to $k\oplus E^* \oplus k$, where $E=\bigoplus \Bbbk a$, where the sum is over all the arrows of the doubled quiver $\Qu \cup \Qu^{\text{op}}$.   

Other examples include graded affine Hecke algebras and degenerate affine Hecke algebras. See \cite{SW1,SW2,SW3,SW4,SW5,sheplerDeformationTheoryHopf2025} for even more general classes of filtered Koszul algebras of this kind. 

\subsection{Structure of the paper}
In Section \ref{sec:premiliminaries} we recall the definition of a {non-homogeneous Koszul algebra} and the construction of the curved dg-algebra structure on the Koszul dual. We also summarise results about linear algebra over semisimple rings and results on homogeneous Koszul duality that are used in the rest of the paper. Since the associated graded algebra plays an important role, we recall a generalization of the PBW theorem (Theorem \ref{thm:PBW}) in the Koszul setting, due to Braverman and Gaitsgory \cite{bravermanPoincareBirkhoffWitt1996}. 

In Section \ref{sec:examples} we describe the explicit examples of symplectic reflection algebras, deformed preprojective algebras, and graded Hecke algebras. The proof of the main results is in Section \ref{sec:equivalences}. After introducing the relevant categories and functors, we prove that $G$ is fully-faithful in Theorem \ref{thm:FGquasiK}. We first prove that $\Ker F=\mathcal{N}$ in the graded case (Corollary \ref{cor:homokoszulcomplex}), then reduce the filtered case to the graded case using a spectral sequence argument in Theorem \ref{thm:quasiisoJF}. We consider the case of bounded complexes and that of finite global dimension in Corollaries \ref{cor:main2abound} and \ref{cor:main2b}. In Section \ref{sec:Positselski} we discuss the relation of our work to that of Positselski and prove a specialization of the main equivalence to the Noetherian case in Corollary \ref{cor:equivalenceNoetherian}. 

Applications and future directions are discussed in Section \ref{sec:applications}. For the reader's convenience, we provide in Appendix \ref{appendix} some results on monoidal algebra that are used throughout the paper.

\subsection*{Acknowledgements}  We are grateful to Matt Booth for useful conversations about Koszul duality. The first author was supported in part by Research Project Grant RPG-2021-149 from The Leverhulme Trust. The first author was also supported by EPSRC grants EP-W013053-1 and EP-R034826-1. The second author acknowledges the financial support of EPSRC (Grant Number EP/V053728/1). The third author is supported by a PhD scholarship from the Carnegie Trust for the Universities of Scotland.

\section{Non-homogeneous Koszul algebras}\label{sec:premiliminaries}

\subsection{Basic definitions and notation}

Throughout, $\Bbbk$ denotes a fixed field. All vector spaces are over $\Bbbk$, all tensor products $\otimes$ without a subscript are over $\Bbbk$. Let $\rin$ be a semisimple algebra over $\Bbbk$ (not necessarily commutative). We always assume that $\rin$ is finite-dimensional over $\Bbbk$. Abusing terminology, a ring $U$ is said to be a $\rin$-algebra if there exists a ring homomorphism $\rin \to U$. We assume that the composition $\Bbbk \to \rin \to U$ has image in $Z(U)$ and all $\rin$-bimodules are assumed to be $\Bbbk$-central. A graded $\rin$-algebra $A$ is a graded ring with a morphism $\rin \to A_0 \subset A$. A graded $k$-algebra is said to be \textit{\coconnective}if $A = \bigoplus_{p \ge 0} A_p$ with $A_0 = \rin$. We let $-A$ denote the opposite of an algebra $A$.  Unless stated otherwise, all modules will be \textit{left} modules. 

By a finitely generated $\rin$-bimodule we mean a bimodule $E$ which is finitely generated on the left and finitely generated on the right. Denote by $T_\rin(E)$ the tensor algebra over $\rin$, defined as 
$$T_\rin(E):=\rin\oplus E\oplus (E\ok E)\oplus (E\ok E\ok E) \oplus \dots $$
We regard $T_\rin(E)$ as a graded $\rin$-algebra, with the standard grading $T_\rin(E)_i=E^{\ok i}$, and as a filtered algebra, with $T_\rin(E)=\bigcup_{p\geq0} \Fi_pT_\rin(E)$, where $\Fi_pT_\rin(E):=\bigoplus_{i\leq p}E^{\ok i}$. 

\begin{rem}
    Here and after, we mean graded/filtered in the category of $k$-bimodules, i.e.\ all homogeneous pieces are $k$-bimodules and $k$ is always in degree zero. 
\end{rem}

\begin{defn}
    A \emph{quadratic algebra} is a connected graded $k$-algebra $A$, generated by the finitely generated $k$-bimodule $E=A_1$ so that $A\cong T_k(E)/I$ as graded algebras, with the defining ideal $I$ generated by its degree two part $I_2$. 
\end{defn}

Since we have assumed that $E$ is finitely generated both as a left and as a right $\rin$-module, each $A_i$ is finite as a left and as a right $\rin$-module.

\begin{defn}
    A \emph{non-homogeneous quadratic algebra} is a filtered $k$-algebra $U=\bigcup_{p\geq0}\Fi_pU$, with $\Fi_0 U=k$, $\Fi_1 U=k\oplus E$, such that $U$ is generated by the finitely generated $k$-bimodule $E$ and $U\cong T_k(E)/J$ as filtered algebras, with the defining ideal $J$ generated by $\Fi_2 T_\rin(E)\cap J$. 
\end{defn}

\begin{rem}\label{rem:associatedquadrated}
    Every non-homogeneous quadratic algebra $U$ has an \emph{associated quadratic algebra}
    $A:=T_\rin(E)/I $, with $I$ the ideal generated by $p_2(\Fi_2T_\rin(E)\cap J)$, where $p_2 \colon T_k(E) \to T_k(E)_2$ is projection to the second degree and $J$ the defining ideal of $U$.
\end{rem}

\begin{defn}\label{defn:Koszulalgebra}
    A \emph{Koszul algebra} is a connected graded $\rin$-algebra $A$ such that the $A$-module $\rin= A / A_{> 0}$ has a graded projective resolution
    \begin{equation}\label{eq:Koszulalgebra}
        \cdots\rightarrow P^{-2}\rightarrow P^{-1}\rightarrow P^0\to \rin \to 0
    \end{equation}
    where each $P^{-i}$ is generated as an $A$-module by its degree $i$ piece.
\end{defn}

It is a well-known fact (see \cite[Proposition 1.2.3]{beilinsonKoszulDualityPatterns1996a}) that any Koszul algebra is quadratic. It is then natural to give the following definition:

\begin{defn}\label{defn:nonhomKoszul}
    A non-homogeneous quadratic algebra is non-homogeneous Koszul if its associated quadratic algebra (as defined in Remark \ref{rem:associatedquadrated}) is Koszul. 
\end{defn}

\subsection{Linear algebra}

The goal of this section is to summarize linear algebra over a semisimple ring $\rin$, analogous to the development in \cite{beilinsonKoszulDualityPatterns1996a}. Let $E$ a finitely generated $\rin$-bimodule. The space $E^* := \Hom_{-\rin}(E,\rin)$ of right $\rin$-linear maps from $E$ to $\rin$ is a $\rin$-bimodule by $(a.f)(x)=af(x)$ and $(f.a)(x)=f(a.x)$. If $M$ is a sub-bimodule of $E$, we denote by $M^\perp$ the sub-bimodule of $E^*$ of functions that are $0$ on $M$. Similarly, define ${}^* E = \Hom_{\rin}(E,\rin)$ to be the space of left $\rin$-linear maps and ${}^{\perp} M$ the sub-bimodule of ${}^* E$ of functions that are $0$ on $M$. The bimodule structure on $^*E$ is $(a \cdot f \cdot b)(e) = f(ea)b$. 

Recall that, if $E_1,E_2$ are $\rin$-bimodules, we have a natural isomorphism $E_1^*\ot_\rin E_2^*\rightarrow(E_2\ot_{\rin}E_1)^*$, given by
\begin{equation}\label{eq:dualtensorsemisimple}
    (f\ot g)(e_2\oR e_1)=f(g(e_2).e_1).
\end{equation}
In particular, $(E\ok E)^*\cong E^*\ok E^*$. 

Similarly, ${}^* E_1 \ot_{\rin} {}^* E_1 \iso {}^*(E_1 \otimes E_2)$ is given by 
\begin{equation}\label{eq:dualtensorsemisimpleleft}
(f \otimes g)(e_2 \otimes e_1) = g\big(e_2 \cdot f(e_1)\big). 
\end{equation}
If $\operatorname{Hom}_{k}(E_1,E_2)$ is the space of left \( k \)-linear maps \( E_1 \to E_2 \) and $\operatorname{Hom}_{-k}(E_1,E_2)$ the space of right \( k \)-linear maps \( E_1 \to E_2 \) then:
\[
\operatorname{Hom}_{-k}(E_1,E_2) \cong E_2 \otimes_k \operatorname{Hom}_{-k}(E_1, k)=E_2\ok E_1^*,
\]
and
\[
\operatorname{Hom}_{k}(E_1, E_2) \cong \operatorname{Hom}_{k}(E_1, k) \otimes_k E_2={}^*E_1\ok E_2.
\]
The evaluation maps are given by:
\[
\operatorname{ev}_E : E \otimes_k E^* \to k, \quad \operatorname{ev}_E(e \otimes f) = f(e),
\]
\[
\widetilde{\operatorname{ev}}_E : {}^*E \otimes_k E \to k, \quad \widetilde{\operatorname{ev}}_E(f \otimes e) = f(e), 
\]
and the coevaluation maps are:
\[
c_E : k \to \operatorname{Hom}_{-k}(E, E)=E\ok E^*, \quad c_E(r)(e) = re,
\]
\[
\widetilde{c}_E : k \to \operatorname{Hom}_{k}(E,E)={}^*E\ok E, \quad \widetilde{c}_E(r)(e) = er.
\]
There are canonical isomorphisms
\[
\varphi \colon E \xrightarrow{\sim} {}^* (E^*), \ \varphi(e)(f) = f(e), \quad \widetilde{\varphi} \colon E \xrightarrow{\sim}  ({}^* E)^*, \ \widetilde{\varphi}(e)(f) = f(e).
\]

\begin{rem}
    If $E$ is a $\Z$-graded $\rin$-bimodule then we always take the graded dual
    \[
    E^* \coloneq \bigoplus_{i \in \Z} (E)^*_i, \quad {}^* E \coloneq \bigoplus_{i \in \Z} ({}^* E)_i,
    \]
    where $(E^*)_i = (E_{-i})^*$ and $({}^* E)_i={}^*( E_{-i})$. The isomorphisms $\varphi$, $\widetilde\varphi$ are isomorphisms of graded $k$-bimodules and the evaluation and coevaluation maps are graded morphisms of $k$-bimodules. 
\end{rem}

\begin{rem}    
    If $A$ is a graded $\rin$-algebra then $A^*$ is a graded $(k,A)$-bimodule and ${}^* A$ is a graded $(A,k)$-bimodule. Here,
    \[
    (r \cdot \phi \cdot a)(b) := r \phi(ab), \quad \forall \, \phi \in A^*, \, r \in \rin, \, a,b \in A. 
    \]
    Since we do not assume $k \subset Z(A)$, there is no natural left $A$-module structure on $A^*$. 
\end{rem}

\subsection{Graded hom and tensor}

We now explain two constructions that will be used throughout the paper. Given two (cohomologically) graded modules $M,N$ over a (cohomologically) graded $\rin$-algebra $B$, one can consider the collection of graded morphisms $M\rightarrow N$. This is naturally a graded $\Bbbk$-vector space, which we denote $\HOM_B(M,N)$. Its $i$-th piece consists of all degree $i$ morphisms of graded $B$-modules $M\to N$.  If $M$ and $N$ come equipped with degree one $k$-linear maps $d_M$ and $d_N$, there is a canonical way to define a degree one morphism on $\HOM_B(M,N)$:
\begin{equation}\label{eq:homdifferential}
(f^i)_{i \in \mathbb{Z}} \mapsto (d_N^i f^i - (-1)^{\lvert f \rvert} f^{i+1} d^i_M)_{i \in \mathbb{Z}}.
\end{equation}
For example, if $M$, $N$ are the total spaces of two complexes of $B$-modules, then $\HOM_B(M,N)$ is naturally a cochain complex. We will also consider more general cases where $d_M$ and $d_N$ do not square to zero. When $M$ and $N$ are complexes, we can also consider the graded-$\Ext$:
$$\underline{\Ext}_B^{\idot}(M,N) := \bigoplus_{n \in \Z} \Ext^{\idot}_{\grLmod{B}}(M[n],N).$$
In general, $\HOM_B(M,N)$ and $\underline{\Ext}_B^{\idot}(M,N)$ are proper subspaces of $\Hom_B(M,N)$ and ${\Ext}_B^{\idot}(M,N)$, respectively.

Similarly, if $M$ is a right graded $B$-module and $N$ is a left graded $B$-module, we can consider the graded tensor product $M\ot_B N$, which is naturally a graded $\Bbbk$-vector space with 
$$(M\ot_B N)^p=\bigoplus_{m+n=p}M^m\ot_B N^n. $$
If $M$ and $N$ come equipped with degree one $k$-linear maps $d_M$ and $d_N$, there is a canonical way to define a degree one morphism on $M\ot_B N$:

\begin{equation}\label{eq:homtensor}
d(m\ot n):=d_M(m)\ot n+(-1)^{|m|}m\ot d_N(n).
\end{equation}

\begin{rem}
    Let $E_1,E_2$ be graded $k$-modules, finitely generated in each degree. Then
    \begin{equation}\label{eq:gradedhomtensordual}
        \HOM_{k}(E_1,E_2)^n\cong \prod_{m\in\Z}{}^*(E_1)^{m}\ok (E_2)^{n-m} 
    \end{equation}
    and similarly for the right linear maps. In general, the product is not a direct sum, and so there is no graded isomorphism
    $$\HOM_{k}(E_1,E_2)\ncong {}^*E_1\ok E_2, \quad \HOM_{-k}(E_1,E_2)\ncong  E_2\ok (E_1)^*. $$
    Under a boundedness condition, the product \eqref{eq:gradedhomtensordual} is finite and so gives the following proposition.
\end{rem}

\begin{prop}\label{prop:gradedkhom}
    Let $E_1,E_2$ be graded $k$-modules, finitely generated in each degree. If $E_1$ is bounded below and $E_2$ is bounded above, or vice versa, there are graded isomorphisms
    $$\HOM_{k}(E_1,E_2)\cong {}^*E_1\ok E_2, \quad \HOM_{-k}(E_1,E_2)\cong  E_2\ok (E_1)^*. $$
\end{prop}

\subsection{Homogeneous Koszul duality}

We recall some standard results about homogeneous Koszul duality. For a more detailed exposition, see \cite{beilinsonKoszulDualityPatterns1996a}.

Recall that $\rin$ is a semisimple algebra over $\Bbbk$. Let $E$ be a finitely generated $\rin$-bimodule and $Q\subset E\ok E$ a sub-bimodule. Consider the quadratic $\rin$-algebra $A= T_\rin(E)/(Q)$. The (right) quadratic dual of $A$ is the $\rin$-algebra $A^!$ defined as
\begin{equation*}
    A^!=T_\rin(E^*)/(Q^\perp),
\end{equation*}
where $Q^\perp$ is the kernel of the restriction map $E^*\ok E^*\rightarrow Q^*$. Clearly, $A^!$ is a quadratic algebra. Similarly, we can define $^!A$ using left-duals instead of right-duals.

\begin{prop}
    If $A$ is Koszul, so are $A^!$ and $^!A$.
\end{prop}
\begin{proof}
    This is \cite[Proposition 2.9.1]{beilinsonKoszulDualityPatterns1996a}.
\end{proof}

Clearly, $^!(A^!)=A$ and $({}^!A)^!=A$, giving a duality of Koszul algebras. We call $^!A$ and $A^!$ the left/right Koszul dual of $A$, respectively.

We now recall the construction of the Koszul complex. Let $Q^{(-1)}=E$, $Q^{(-2)}=Q$, and
$$Q^{(-i)}:=\bigcap_{j=0}^{i-2} E^{\ot j}\ok R\ok E^{\ot i-j-2}\subset E^{\ot i}, \quad i \geq 2. $$
Define a complex of $A$-modules $\mathcal{K}^\bullet(A)$ as
\[
\dots\rightarrow A\ok Q^{(-2)}\rightarrow A\ok Q^{(-1)}\rightarrow A,
\]
with differential $d(a\ok (x_1\ok\cdots \ok x_i)):=(ax_1)\ok(x_2\ok\cdots\ok x_i)$ and with $A$ in degree 0. This is a complex of graded modules, meaning that there is an ``internal grading", compatible with the differentials. To avoid confusion, we denote the cohomological grading as $\mc{K}(A)^p$ and the internal grading as $\mc{K}(A)_q$. Then  
\begin{equation}\label{eq:internalgrading}
    \mc{K}(A)_q := \bigoplus_{m \in \Z} A_{q+m} \otimes Q^{(m)} = \bigoplus_{m-n=q} A_{m} \otimes Q^{(n)}, \quad \mc{K}(A)^p_q=A_{p+q}\ok Q^{(p)}. 
\end{equation}
These conventions are fixed so that the cohomological and internal grading on $\mc{K}(A)$ match that of \eqref{eq:Koszulalgebra} i.e. $\mc{K}(A)$ is concentrated in non-positive cohomological degrees and $\mc{K}(A)^{-i}$ is generated by its part of (internal) degree $i$. We call $\mathcal{K}^\bullet(A)$ the \emph{Koszul complex} of $A$.

The following is \cite[Theorem~2.6.1]{beilinsonKoszulDualityPatterns1996a}. 

\begin{thm}\label{thm:koszulprojresolution}
Let $A$ be a quadratic $\rin$-algebra. The Koszul complex is a minimal projective resolution of $\rin$, viewed as a graded left $A$-module concentrated in degree zero, if and only if $A$ is Koszul. 
\end{thm}

We can equivalently define the Koszul complex in terms of the Koszul dual. 

\begin{lem}\label{lem:!A!}
Let $A$ be a Koszul algebra. Then, for all $i \ge 0$, there are isomorphisms of $\rin$-bimodules $(A^!)^i\cong\left({Q^{(-i)}}\right)^*$ and $\left({}^! A\right)^i \cong {}^* \left({Q^{(-i)}}\right)$. 
\end{lem}

\begin{proof}
For all $i\geq2$, we have:
    \begin{equation*}
               (A^!)^i= ( E^*)^{\ot i}/\sum_{j=0}^{i-2} ({E}^*)^{\ot (i-2-j)}\ok {Q}^\perp \ot ( E^*)^{\ot j}.
    \end{equation*}
    Since $\sum_{j=0}^{i-2} ({E}^*)^{\ot (i-2-j)}\ok {Q}^\perp \ot ( E^*)^{\ot j}$ is the kernel of the restriction $(E^{\ot i})^*\rightarrow (Q^{(-i)})^*$ we can identify $(A^!)^i\cong({Q^{(-i)}})^*$, for all $i$. Similarly, ${}^! A^i\cong{}^*({Q^{(-i)}})$.
\end{proof}

The total space of the Koszul complex is thus isomorphic to $\mathcal{K}'(A):=A\ok {}^*(A^!)$ and to $\mathcal{K}''(A):=A\ok ({}^!A)^*$. These identifications preserve the bigradings
\begin{align*}
   \mc{K}'(A)^p=A\ok{}^*((A^!)^{-p}), \quad &\mc{K}'(A)_q=\bigoplus_{m\in\Z} A_{q+m}\ok{}^*((A^!)^{-m}), \\
\mc{K}''(A)^p=A\ok({}^!A_{-p})^*, \quad &\mc{K}''(A)_q=\bigoplus_{m\in\Z} A_{q+m}\ok({}^!A_{-m})^*,    
\end{align*}

Define the differential
\[
d':A\ok {}^* ((A^!)^i) = A \ok k \ok {}^* ((A^!)^i) \xrightarrow{c_E}A\ok E \ok E^\ast \ok {}^* ((A^!)^i) \xrightarrow{m_A \ot a_{A^!}} A\ok {}^* ((A^!)^{i-1}),
\]
where $c_E$ is the coevaluation map $k \to E \ok E^\ast$, $m_A$ multiplication in $A$ and $a_{A^!}$ the restriction to $E^* \subset A^!$ of the action map $a_{A^!} \colon A^! \ot_{\rin} {}^* (A^!) \to {}^* (A^!)$ given by $a_{A^!}(b \ot f)(x) = f(xb)$. Define also 
\[
d'':A\ok (({}^!A)^i)^* \cong \Hom_{-k}(({}^!A)^i,A)\xrightarrow{}\Hom_{-k}(({}^!A)^{i-1},A)\cong A\ok (({}^!A)^{i-1})^*
\]
where, for $f\in\Hom_{-k}(({}^!A)^i,A)$, 
\[
d''(f):(({}^!A)^{i-1})^*\xrightarrow{\widetilde{c}_E}(({}^!A)^{i-1})^*\ok {}^*E\ok E\xrightarrow{ m_{^!A}\ok1}(({}^!A)^{i})^*\ok E\xrightarrow{f\ok1}A\ok E \xrightarrow{m_A} A.
\]
Explicitly, if $c_E(1):=\sum_\alpha x_\alpha\ok\hat{x}_\alpha$ and $\widetilde{c}_E(1):=\sum_\alpha \check{x}_\alpha\ok x_\alpha$, then
$$d'(a\ok f)=\sum_\alpha ax_\alpha\ok(\hat{x}_\alpha.f), \quad d''(f)(a)=\sum_\alpha f(a\check{x}_\alpha)x_\alpha. $$

\begin{prop}\label{prop:homokoszulcomplex}
    The Koszul complex $(\mathcal{K}(A),d)$ is isomorphic, as a complex of graded $A$-modules, to both $(\mathcal{K}'(A),d')$ and $(\mathcal{K}''(A),d'')$.
\end{prop}
\begin{proof}
    Let $a\in A$, $q_1\ok\cdots\ok q_i\in Q^{(-i)}$ and $f_i\ok\cdots\ok f_2\in (A^!)^{i-1}$. Then
\begin{align*}
    (d'(a\ok q_1\ok\cdots\ok q_i))(f_i\ok\cdots\ok f_2)&=\sum_\alpha ax_\alpha\ok \hat{x}_\alpha.(q_1\ok\cdots\ok q_i)(f_i\ok\cdots\ok f_2)\\
    &=\sum_\alpha ax_\alpha\ok (q_1\ok\cdots\ok q_i)(f_i\ok\cdots\ok f_2\ok \hat{x}_\alpha)\\
    &=\sum_\alpha ax_\alpha\ok (f_i\ok\cdots\ok f_2\ok \hat{x}_\alpha)(q_1\ok\cdots\ok q_i)\\
    &=\sum_\alpha ax_\alpha\ok (f_i\ok\cdots\ok f_2)(\hat{x}_\alpha(q_1).q_2\ok\cdots\ok q_i)\\
    &=\sum_\alpha ax_\alpha\ok (\hat{x}_\alpha(q_1).q_2\ok\cdots\ok q_i)(f_i\ok\cdots\ok f_2).
\end{align*}    
Hence 
$$d'(a\ok q_1\ok\cdots\ok q_i)=\sum_\alpha ax_\alpha\ok \hat{x}_\alpha(q_1).q_2\ok\cdots\ok q_i=aq_1\ok q_2\ok\cdots\ok q_i,$$
so $(\mathcal{K}(A),d)\cong(\mathcal{K}'(A),d')$. Similarly:

\begin{align*}
    (d''(a\ok q_1\ok\cdots\ok q_i))(f_i\ok\cdots\ok f_2)&=\sum_\alpha a(q_1\ok\cdots\ok q_i)(f_i\ok\cdots\ok f_2\ok \check{x}_\alpha)x_\alpha\\
    &=\sum_\alpha a(f_i\ok\cdots\ok f_2\ok \check{x}_\alpha)(q_1\ok\cdots\ok q_i)x_\alpha\\
    &=\sum_\alpha a\check{x}_\alpha(q_1.(f_i\ok\cdots\ok f_2)(q_2\cdots\ok q_i))x_\alpha\\
    &=aq_1(q_2\cdots\ok q_i))(f_i\ok\cdots\ok f_2),
\end{align*}
hence $(\mathcal{K}(A),d)\cong(\mathcal{K}''(A),d'')$.
\end{proof}

\begin{rem}\label{rem:rightkoszulcomplex}
    In the same way, we can describe three isomorphic complexes of graded right $A$-modules that give a projective resolution of $\rin$ as a graded right $A$-module:
    $$(\widetilde{\mathcal{K}}(A),d)=Q^\bullet\ok A, \quad (\widetilde{\mathcal{K}}'(A),d')=({}^!A)^*\ok A, \quad (\widetilde{\mathcal{K}}''(A),d'')={}^*(A^!)\ok A, $$
    where,
    $$d(q_i\ok\cdots\ok q_1\ok a)=q_i\ok\cdots\ok q_2\ok q_1a, \quad d'(f\ok a)=\sum_\alpha f\check{x}_\alpha\ok x_\alpha a,$$ and $(d''(\phi))(v)=\sum_\alpha x_\alpha \phi(v\hat{x}_{\alpha}) $, for $v\in A^!$ and $\phi\in\Hom_{k}(A^!,A)={}^*(A^!)\ok A$.
\end{rem}

\begin{example}
    The standard example of a Koszul algebra is $\Sym(V)$, the symmetric algebra of a finite-dimensional vector space $V$ over a field $\rin = \Bbbk$. The Koszul resolution is
    $$
    \cdots\rightarrow\Sym(V)\ot  \Lambda^2 V\rightarrow \Sym(V) \ot V\rightarrow \Sym(V)\twoheadrightarrow k.$$
    In this case $Q= \Lambda^2 V=\{x\ot y-y\ot x \ | \ x,y \in V \}$ and $Q^\perp=\{f\ot g+g\ot f \ | \ f,g \in V^* \}$, so that $\Sym(V)^!=\Lambda(V^*)$, the exterior algebra of $V^*$.
\end{example}

The following is only used in Proposition~\ref{prop:Koszuldualpreproj}. We postpone the proof to Section~\ref{sec:PropExtproof}; see also \cite[Theorem 2.10.1]{beilinsonKoszulDualityPatterns1996a}.

\begin{prop}\label{prop:koszuldualext}
     Let $A$ be a Koszul $k$-algebra. There are isomorphisms of graded $k$-algebras 
     $$
     {}^! A \cong \Ext_{A}^{\idot}({}_A k,{}_A k)^{\mathrm{op}}, \quad A^! \cong \Ext_{A}^{\idot}(\rin_A ,\rin_A),
     $$
     where ${}_A \rin$ and $\rin_A$ mean $\rin$ considered as a left and right module, respectively.
\end{prop}

\subsection{PBW Theorem for non-homogeneous Koszul algebras}

The following PBW criterion for non-homogeneous Koszul algebras is due to Braverman and Gaitsgory \cite{bravermanPoincareBirkhoffWitt1996}. It is not needed for the proof of our main results, but has many practical applications. For example, it gives a concrete criterion for checking if a finitely presented $\rin$-algebra is non-homogeneous Koszul; see Remark~\ref{rem:PBW}.

As in the previous subsection, let $E$ be a finitely generated $k$-bimodule. Let $R\subset \Fi_2 T_k(E)$ be a sub-$k$-bimodule such that $\Fi_1 T_k(E)\cap R=0$. Let $Q:=p_2(R)$ be its projection to the degree $2$ part and consider the non-homogeneous quadratic algebra $U:=T_k(E)/(R)$ and the quadratic algebra $A:=T_k(E)/(Q)$. We refer to $R$ and $Q$ as the spaces of relations.

\begin{rem}
    Notice that the algebra $A$ will not, in general, equal the graded quadratic algebra associated to $U$ in Remark \ref{rem:associatedquadrated}. They are equal only if  $p_2(\mb{F}_2T_\rin(E)\cap(R))=Q$, or equivalently if $\mb{F}_2T_\rin(E)\cap(R)=R$. We say that $A$ is the quadratic algebra associated to the space of relations $R$. It depends on the chosen space of relation, and is not an intrinsic property of the algebra. See also Remark~\ref{rem:PBW}.
\end{rem}

\begin{defn}
    The space of relations $R$ is of \emph{Poincaré-Birkoff-Witt (PBW)} type if the natural graded epimorphism $A \twoheadrightarrow \gr (U)$ is an isomorphism of graded algebras.
\end{defn}

Since $R \cap (\rin \oplus E)$ is zero, there are $\rin$-bilinear maps
$\alpha:Q\rightarrow E$ and $\beta:Q\rightarrow \rin$
such that
\begin{equation}\label{eq:R}
    R = \{x + \alpha(x) + \beta(x) \,| \,x \in Q\}.
\end{equation}

Consider the following conditions on $\alpha$ and $\beta$, where all the maps are considered as $(E\ok Q) \cap (Q\ok E) \to E\ok E$:
    \begin{enumerate}
        \item $\im(\alpha \ok\id-\id\ok\alpha)\subset Q$.
        \item $\alpha\circ(\alpha\ok\id-\id\ok\alpha)=\beta\ok\id-\id\ok \beta$.
        \item $\beta\circ(\alpha\ok\id-\id\ok\alpha)=0$.
    \end{enumerate}
Conditions $(1)-(2)-(3)$ can alternatively be rephrased as: 
\[
\alpha \ok\id-\id\ok\alpha \colon (Q \otimes_k E) \cap (E \otimes_k Q) \to E \otimes_k E
\]
factors through $Q \subseteq E \otimes_k E$ and the following diagram commutes
\[
\begin{tikzcd}
& (Q \otimes_k E) \cap (E \otimes_k Q) \arrow[dl,bend right = 10,"0", swap] \arrow[dr,bend left=10,"\beta \otimes E - E \otimes \beta"]   \arrow[d,"\alpha \otimes E - E \otimes \alpha"] & \\
k & \arrow[l,"\beta"] Q \arrow[r,"\alpha",swap] & E 
\end{tikzcd}
\]

\begin{thm}\label{thm:PBW}
    Let $R\subset \Fi_2 T_k(E)$ with $R\cap \Fi_1 T_k(E)=0$. We have the following chain of implications:
    \begin{enumerate}[(i)]
        \item The space $R$ is of PBW-type.
        \item $(R)\cap \Fi_2 T_k(E)=R$.
        \item The maps $\alpha,\beta$ satisfy conditions $(1)-(2)-(3)$.
    \end{enumerate}
    Moreover, if $A=T_k(E)/(Q)$ is Koszul, then $(iii)\implies (i)$, hence $(i)-(ii)-(iii)$ are equivalent.
\end{thm}
\begin{proof}
    This was proved in {\cite{bravermanPoincareBirkhoffWitt1996}} for algebras over a field, but the same proof works for $\rin$ a semisimple algebra over $\Bbbk$. We sketch the proof for the reader's convenience.

    \emph{(i)$\implies$(ii)} Assume that $R$ is of PBW-type and let $I$ be the ideal generated by $R$. Let $U:=T_k(E)/I$. Then the natural surjection $p:A\rightarrow\gr(U)$ is an isomorphism of graded algebras. In particular, 
    \begin{equation*}
        p_2:A_2=\Fi_2 T_k(E)/Q \iso\gr(U)_2= \Fi_2 T_k(E) /(\Fi_1 T_k(E) +\Fi_2T_k(E)\cap I), 
    \end{equation*}
    which implies $Q\cong \Fi_1 T_k(E) +\Fi_2T_k(E)\cap I$. On the other hand, $Q=p_2(R)$, so $\Fi_2 T_k(E)\cap I=R$ modulo $\Fi_1T_k(E)$. Since $p_1:E\rightarrow \Fi_1T_k(E)/(\Fi_0T_k(E)+\Fi_1T_k(E)\cap I)$ is an isomorphism, $I\cap \Fi_1T_k(E)=0$. Thus, $\Fi_2T_k(E)\cap I=R$. 

    \emph{(ii)$\implies$(iii)} Tensoring equation \eqref{eq:R} on the left and on the right with $E$, we get the following relation for every $x\in(Q \ok E) \cap (E \ok Q)$:
    \begin{equation*}
        x+(\id\ok\alpha) (x) +(\id\ok\beta) (x)=0=x+(\alpha\ok\id) (x)+(\beta\ok\id) (x) \quad (\mod\,  I),
    \end{equation*}
    hence
    \begin{equation*}
        (\alpha\ok\id-\id\ok\alpha)(x)+(\beta\ok\id -\id\ok\beta)(x)\in I\cap F_2(E)=R.
    \end{equation*}
    Applying $p_2$ to both sides we get $(\alpha\ok\id-\id\ok\alpha)(x)\in Q$. Then, by definition of $\alpha$ and $\beta$,
    $$\alpha\circ(\alpha\ok\id-\id\ok\alpha\ok\id)(x)=(\beta\ok\id -\id\ok\beta)(x), \quad \quad \beta\circ(\alpha\ok\id-\id\ok\alpha)(x)=0. $$

    The last implication is more complicated. The idea is to use the theory of graded deformations, and it goes as follows. We want to show that, if conditions $(1)$ to $(3)$ are satisfied and if $A$ is Koszul, it is always possible to construct a graded deformation of $A$ whose fibre at $t=1$ is isomorphic to $U$ as a filtered algebra. 
    One can show that conditions $(1)$ to $(3)$ guarantee the existence of a first, second, and third degree deformation of $A$, respectively. Up to this point, the condition of $A$ Koszul is not needed. If $A$ is Koszul, then every third-degree graded deformation can be extended to a full graded deformation. 

    Let $\widetilde U$ be the fibre at $t=1$ of this graded deformation. By the theory of graded deformations, there is a graded algebra isomorphism $\gr(\widetilde{U})\cong A$. One can show there is a filtered map $\phi:{U}\rightarrow \widetilde{U}$ such that the composition
    $$\begin{tikzcd}
A \arrow[r, "p"] & \gr({U}) \arrow[r, "\gr(\phi)"] & \gr(\widetilde{U}) \arrow[r] & A
\end{tikzcd} $$
    is the identity. In particular, this implies that $p, \phi$ are isomorphisms. For more details, see \cite{bravermanPoincareBirkhoffWitt1996}.    
\end{proof}

\begin{rem}\label{rem:PBW}
        Every non-homogeneous quadratic algebra $U$ has a canonical space of relations given by $R_{\text{max}}:=J\cap \Fi_2 T_k(E)$, where $J$ is the defining ideal of $U$. This is the maximal space of relations that generates $J$. Notice that $J\cap \Fi_1 T_k(E)=0$ because the quotient map $T_k(E)/J\rightarrow U$ is an isomorphism between $\Fi_1 T_k(E)$ and $\Fi_1 U$. Since $R_{\text{max}}$ satisfies condition $(ii)$ of Theorem \ref{thm:PBW}, if $U$ is non-homogeneous Koszul then the graded $\rin$-algebra $\gr(U)$ is always isomorphic to the associated quadratic algebra of~$U$.

To check if a $\rin$-algebra is non-homogeneous Koszul, one needs to know its associated quadratic algebra. However, in practice one is often given an algebra $U$ via its presentation by generators and relations. The space of relations $R$ may be strictly smaller than $R_\text{max}$. In this situation, computing the space of relations $Q$ is straightforward, but it might not be possible to describe $R_\text{max}$ or the associated quadratic algebra $A$. Condition $(iii)$ of Theorem \ref{thm:PBW} gives a direct way to check if we have $R=R_\text{max}$.
\end{rem}

\subsection{Non-homogeneous duality} 

In this section, we explain how to construct the Koszul dual of a non-homogeneous Koszul algebra (Definition~\ref{defn:nonhomKoszul}). This was first introduced by Positselski~\cite{positselskiiNonhomogeneousQuadraticDuality1993}. 

\begin{defn}
    A $\mathbb{Z}$-graded $\rin$-algebra $\Lambda = \bigoplus_{i \in \Z} \Lambda^i$ is a \emph{curved differential graded algebra (cdga)} if there exist a $\rin$-bilinear map $d:\Lambda\rightarrow \Lambda$ of degree $1$ and an element $c\in \Lambda^2$ such that
    \begin{enumerate}[(i)]
    \item $d(a b) = d(a) b + (-1)^{|a|} a d(b)$ for $a,b \in \Lambda$ homogeneous;
        \item $d^2=[c,\cdot]$;
        \item $d(c)=0$.
    \end{enumerate}
\end{defn}

We note that the above commutator is the graded version so that $[a,b] = ab - (-1)^{\lvert a \rvert \lvert b \rvert}ba$ for homogeneous elements $a,b$.

\begin{defn}
    A morphism of cdgas $(\Lambda,d,c),(\Omega,d',c')$ is a pair $(\phi,a)$ where $\phi:\Lambda\rightarrow\Omega$ is a morphism of graded algebras and $a$ is an element of $ \Omega^1$ such that
    \begin{enumerate}
        \item $d'(\phi(x))=f(dx)+[a,f(x)]$ for all $x\in\Lambda$; and
        \item $c'=f(c)+d'a-a^2$.
    \end{enumerate}
\end{defn}

We say that a cdga is Koszul if its underlying algebra is Koszul.

\begin{thm}\label{thm:cdga}
    Let $U$ be a non-homogeneous Koszul $\rin$-algebra, with defining ideal $J$. Let $R=J\cap \Fi_2(E)$, $Q=p_2(R)$, $A=T_\rin(E)/(Q)=\gr(U)$, and $\alpha,\beta$ the maps defined in \eqref{eq:R}. Then $(A^!,\alpha^\ast,\beta)$ is a cdga over $\rin$. 
\end{thm}
\begin{proof}
    By Theorem \ref{thm:PBW}, the maps $\alpha$ and $\beta$ satisfy relations $(1)-(2)-(3)$. Notice that, by definition, $(A^!)^1=E^*$ and $(A^!)^2=Q^*$. Hence, $\beta \colon Q \to \rin$ is an element in $Q^* = (A^!)^2$ and $\alpha^*:(A^!)^1\rightarrow (A^!)^2$. The cdga structure on $A^!$ then follows from dualizing the properties $(1)$-$(3)$ of Theorem \ref{thm:PBW}. From $(1)$, the map $d:=\alpha^*$ has a well defined extension to a $\rin$-bilinear degree $1$ derivation on $A^!$ by the Leibniz rule. Dualizing $(2)$, we get that $d^2=[\beta,\cdot]$. Finally, $(3)$ implies that $d(\beta)=0$.
\end{proof}

\begin{rem}
    Note that $c$ is graded central if and only if $d^2 = 0$ on $A^!$. In particular, if $U$ is non-homogeneous Koszul, then $c$ is graded central if $d = 0$ or, equivalently, if $\alpha= 0$. 
\end{rem}

Let $\mathsf{NHKA}$ be the category of non-homogeneous Koszul algebras, where morphisms are filtered morphisms of $\rin$-algebras, and $\mathsf{CDGKA}$ be the category of cdgas whose underlying algebra is Koszul, with cdga morphisms. Koszul duality can be upgraded to a contravariant functor $\mathsf{NHKA}\rightarrow\mathsf{CDGKA}$ by defining it on morphisms using the natural isomorphisms
$$\{f\in \Hom_\rin(\rin\oplus E_1,\rin\oplus E_2) \,|\, f|_{\rin}=\id \}\cong \Hom_{\rin}(E_1,E_2)\oplus E_1^*. $$

\begin{thm}\label{thm:equivalenceoffilteredcurved}
    The Koszul duality functor is an antiequivalence of categories $\mathsf{NHKA}\iso\mathsf{CDGKA}$.
\end{thm}
\begin{proof}
    This is proved in \cite[Section 3.3]{positselskiiNonhomogeneousQuadraticDuality1993}. See also \cite[Corollary~4.20]{positselskiRelativeNonhomogeneousKoszul2021}. 
\end{proof}

\begin{remark}

While in homogeneous Koszul duality the objects on both sides are the same, namely quadratic algebras, in non-homogeneous Koszul duality the symmetry breaks down. In fact, the objects on the two sides live in very different categories. This is because non-homogeneous Koszul duality is actually a special instance of a wider phenomenon. 

The more general form of Koszul duality is between differential graded algebras and curved differential graded coalgebras. See, for example, \cite{positselskiDifferentialGradedKoszul2023}, or even more generally \cite{booth2024globalkoszulduality}. From this more general point of view, the restriction to Koszul algebras means that the Koszul dual admits a particularly nice presentation as the non-homogeneous quadratic dual.
\end{remark}

\begin{example}
    The simplest example is taking $U=A$ to be a quadratic algebra. Then $\alpha=\beta=0$ and the differential and curvature vanish, leaving just the graded algebra $A^!$. This is the usual Koszul duality.
\end{example}

\section{Motivating Examples}\label{sec:examples}

\subsection{Basic examples}

\begin{example}
    Take $\rin$ to be a field $\Bbbk$ and $E = V$ a finite-dimensional $\Bbbk$-vector space. If $V$ has a symplectic form $\omega \in \Lambda^2 V^*$ then the associated Weyl algebra is defined to be $U = T_\Bbbk V/(R)$ with 
 \[
 R = \{ x \otimes y - y \otimes x - \omega(x,y) \ | \ x,y \in V \}. 
 \]
    The space of relations $R$ is of PBW type with $\gr(U)\cong\Sym(V)$, hence the Weyl algebra is non-homogeneous Koszul. Since $\alpha=0$ and $\beta=-\omega$, the Koszul dual is $\Lambda V^*$ with zero differential and curvature $-\omega\in\Lambda^2V^*$. As above, we note that $\omega$ is central.
\end{example}
\begin{example}
    Let $\rin=\Bbbk$ and $E=\mf g$ a finite-dimensional Lie algebra. The universal enveloping algebra $U=U(\mf g)$ has PBW relations $$R=\{x\ot y-y\ot x-[x,y] \ | \ x,y \in \mf g \},$$ 
    with $\gr(U(\mf g))\cong\Sym(\mf g)$. Hence $U(\mf g)$ is non-homogeneous Koszul. We have $\beta=0$ and $\alpha(x\ot y-y\ot x)=-[x,y]$. The Koszul dual of $U$ is $\Lambda \mf g^*$, with non-trivial differential $\alpha^*$ and zero curvature. That is, $A^!$ is a differential graded algebra equal to the standard cohomological complex of the Lie algebra $\mf g$. 

    More generally, we can take $\beta:\Lambda^2\mf g\rightarrow k$ a $2$-cocycle and form the deformation $U_\beta(\mathfrak{g})$ considered by Sridharan \cite{Sridharan}, with relations
    $$R_\beta=\{x\ot y-y\ot x-[x,y]-\beta(x,y) \ | \ x,y \in \mf g \}.$$
    Here the differential on $\Lambda \mf g^*$ is the same, but it no longer squares to zero. In fact, we have non-zero curvature $c = -\beta\in\Lambda^2\mf g^*$. 
\end{example}

\subsection{Koszul duality for the Symplectic Reflection Algebra}\label{sec:symplectic}

Let $(V,\omega)$ be a finite-dimensional complex symplectic vector space. Recall that an element $s\in \mathrm{Sp}(V)$ is called a \emph{symplectic reflection} if $\rk (\id-s)=2$. Let $G\subset \mathrm{Sp}(V)$ be a finite subgroup and let $\mathcal{S} \subset G$ be the symplectic reflections in $G$. We do not need to assume that $G$ is generated by $\mathcal{S}$. For $v\in V$ and $g\in G$, we denote the action $g.v$ as $v^g$. Similarly, if $f\in V^*$ we write $f^g$ for $g.f:=f(g^{-1}\boldsymbol{\cdot})$. Let $t\in\mathbb{C}$ and let $c:\mathcal{S}\rightarrow \mathbb{C}$ be a conjugation invariant function. For any $s\in\mathcal{S}$, denote by $\omega_s$ the $2$-form that is equal to $\omega$ when restricted to $\im(1-s)$ and $0$ when restricted to $\Ker(1-s)$. The symplectic reflection algebra $H_{t,c}(G)$ is the following quotient of the skew-ring $T(V)\rtimes G$
\begin{equation*}
    H_{t,c}(G):=T(V)\rtimes G /( \mathcal{R}), 
\end{equation*}
where
\begin{equation*}
    \mathcal{R} = \Span_{\Bbbk} \left\langle x\otimes y-y\otimes x-t\omega(x,y)-\sum_{s\in\mathcal{S}}c_s\omega_s(x,y)s \, \mid \, x,y\in V\right\rangle.
\end{equation*}

Let $\rin:=\mathbb{C} G$ and $E:=V\otimes_{\C} \rin$, which is a $\rin$-bimodule with actions:
\begin{equation*}
    g.(v\otimes h)=v^g\otimes gh \quad \quad (v\otimes h).g=v\otimes hg, \quad \forall v\in v, \forall g,h\in G.
\end{equation*}
Consider the tensor algebra $T_\rin(E)$. As a vector space, this is naturally isomorphic to $T(V)\otimes \rin$. The product is given by $g\cdot v=(1\ot g)\ok(v\otimes1)=1\ok g.(v\otimes 1)=1\ok(v^g\otimes g)=v^g\cdot g$, for all $v\in V$ and $g\in G$, which implies that $T_\rin(E)\cong T(V)\rtimes G$ as $\rin$-algebras. We can thus view $H_{t,c}(G)$ as the non-homogeneous quadratic algebra $T_\rin(E)/(R)$ where $R$ is the $k$-bimodule generated by $\mathcal{R}$. The quadratic part of $R$ is 
\begin{equation*}
    Q=E\wedge E:= \Span_{\Bbbk} \langle x\otimes y\otimes g-y\otimes x\otimes g \, | \, g \in G, x,y \in V \rangle,
\end{equation*}
thus the quadratic algebra associated to $R$ is $A:=T_\rin(E)/(E\wedge E)\cong \Sym(V)\rtimes G$.

\begin{thm}[{\cite[Theorem 1.3]{etingofSymplecticReflectionAlgebras2002}}]\label{thm:PBW_EG}
    The symplectic reflection algebra is a non-homogeneous Koszul algebra with
    $\gr(H_{t,c}(G))\cong \Sym(V)\rtimes G.$
\end{thm}

\begin{proof}
    The algebra $A=\Sym(V)\rtimes G$ is Koszul. It suffices then to prove that $R$ satisfies $(iii)$ of Theorem \ref{thm:PBW}, which was proved in \cite[Theorem 1.3]{etingofSymplecticReflectionAlgebras2002}. 
\end{proof}

Consider now the $\rin$-bimodule $E^*$. As a vector space, $E^*\cong \rin\otimes V^* $, because any right $k$-linear map from $E\cong V\ot k$ is uniquely determined by its value on $V$. Explicitly, 
\begin{equation*}
    (g\otimes f)(v\ot 1):=gf(v) \quad \quad g\in G, f\in V^*, v\in V.
\end{equation*}
Under this isomorphism, the canonical $\rin$-bimodule structure on $E^*\cong k\ot V^*$ becomes
\begin{equation*}
    (a.(g\otimes f))(v\otimes1)=agf(v)=(ag\otimes f)(v\otimes 1), 
\end{equation*}
\begin{equation*}
    ((g\otimes f).a)(v\otimes1)=gf(a.v\otimes1)=gf(v^a\otimes a)=(ga)f(v^a)=(ga\otimes f^{a^{-1}})(v\otimes1).
\end{equation*}
In particular, the product in $T_\rin(E^*)$ is given by
$$a\cdot f=a\ok (1\ot f)=1\ok a.(1\ot f)=a \ot f=(a\ot f)\ok 1= (1\ot f^a)\ok a=f^a\cdot a. 
$$
This implies that $T_\rin({E}^*)\cong G\ltimes T(V^*)$, as $\rin$-algebras.

\begin{rem}
    Under this identification, one should write an element in $E^*$ as a linear combination of elements in $(V^{\otimes n})^*$ with coefficients in $\rin$, \emph{with the coefficients on the left}.
\end{rem}

\begin{prop}\label{prop:KosuzldualSRA}
    The Koszul dual to the symplectic reflection algebra $H_{t,c}$ is 
    \begin{equation*}
        A^!:= G \ltimes \left( \bigwedge V^* \right),
    \end{equation*}
    whose cdga structure has trivial differential and curvature $c=-\kappa\in (E\wedge E)^*$ defined as
    \begin{equation*}
        \kappa:=t\omega+\sum_{s\in \mathcal{S}}c_s s\omega_s.
    \end{equation*}
\end{prop}
\begin{proof}
    Let $f\in (E\ok E)^*$, and suppose $f(E\wedge E)=0$. Since $f$ is right $\rin$-linear, this means that $f(x\otimes y\otimes 1-y\otimes x\otimes 1)=0$, for all $x,y\in V$. Thus $Q^\perp=G\ltimes\Span_\C \langle f\ot g+g\ot f \ | \ f,g \in V^*\rangle $. Hence $$A^!=G\ltimes T(V^*)/({Q}^\perp )\cong G\ltimes\left( \bigwedge V^* \right) .$$ 
    The cdga structure follows from Theorem \ref{thm:cdga} since, in our case, $\alpha=0$ and $\beta=-\kappa$. 
\end{proof}

\subsection{Graded Hecke algebras}

We can generalize the example of symplectic reflection algebras by considering \textit{graded Hecke algebras}, as introduced by Drinfeld \cite{DrinfeldGradedHecke}. Namely, given a finite-dimensional $\Bbbk$-vector space $V$ and finite group $G \subset GL(V)$ such that $\Bbbk G$ is semisimple, we pick $a_g \colon V \times V \to \Bbbk$ skew-symmetric, set $\rin = \Bbbk G$ and $E = V \ot_{\Bbbk} \rin$ as above. Then $U$ is the non-homogeneous quadratic algebra $T_{\rin} E / ( R)$, where $R$ is the $\rin$-sub-bimodule of $E \ot_{\rin} E$ generated by 
\[
\Span_{\Bbbk} \left\langle x\otimes y-y\otimes x- \sum_{g \in G} a_g(x,y) g \, |\, x,y\in V\right\rangle.
\]
As in the case of symplectic reflection algebras, the homogeneous quadratic algebra associated to $R$ is $G \ltimes \Sym(V)$, which is Koszul. Drinfeld \cite{DrinfeldGradedHecke} gave a criterion for when $U$ satisfies the PBW property; see \cite[Theorem~1.9]{RamShepler} for a more complete statement and full proof. We note that the PBW property forces $a_g = 0$ for all $g$ such that $\mathrm{codim}_V V^g \neq 2$. Just as in Proposition~\ref{prop:KosuzldualSRA}, when the PBW property holds, and thus $U$ is non-homogeneous Koszul, the Koszul dual equals $\left( \bigwedge V^* \right) \rtimes G$ as an algebra and $c = -\sum_{g \in G} g a_g$. The differential is trivial. 

In \cite{LusztigGradedHecke}, Lusztig introduced a graded version of the affine Hecke algebra associated to any Weyl group $G$ (often called the degenerate affine Hecke algebra in the literature). The usual presentation of these algebras does not realise them as non-homogeneous quadratic algebras. However, it is shown in \cite[Theorem~3.5]{RamShepler} that they are isomorphic to graded Hecke algebras in the above sense, where $V$ is taken to be the reflection representation for $G$. Therefore, it is possible to apply Koszul duality to understand the derived category of modules over these graded Hecke algebras. 

In a series of papers, including \cite{SW1,SW2,SW3,SW4,SW5} and most comprehensively \cite{sheplerDeformationTheoryHopf2025}, Shepler, Witherspoon and coauthors have generalized Drinfeld's deformations of skew group rings to PBW deformations of smash products $S \# H$, where $S$ is a Koszul algebra acted on by a Hopf algebra $H$. Assuming that $H$ is finite-dimensional and semisimple, we expect that their construction gives rise to a very broad class of examples of non-homogeneous Koszul algebras. It would be interesting to compute the Koszul dual in this generality. 

\subsection{The preprojective algebra}

Let $\Qu$ be a finite connected quiver whose underlying graph is not of finite ADE type. If $a \in \Qu_1$ is an arrow, then $t(a), h(a) \in \Qu_0$ denote the tail and head of $a$, respectively. We compose arrows as functions. Thus, $a b = 0$ unless $h(b) = t(a)$. Let $\overline{\Qu} = \Qu \cup \Qu^{\op}$ be the doubled quiver. For each $a \in \Qu_1$, there is a (unique) opposite arrow $a^* \in \Qu_1^{\op}$. Let $\rin = \bigoplus_{i \in \Qu_0} \Bbbk e_i$ be the vertex subalgebra of the path algebra of $\overline{\Qu}$ with $e_i e_j = \delta_{i,j} e_i$ and 
\[
E = \bigoplus_{a \in \overline{\Qu}} \Bbbk a = \left( \bigoplus_{a \in \Qu_1} \Bbbk a \right) \oplus \left( \bigoplus_{a^* \in \Qu_1^{\op}} \Bbbk a^* \right) = D \oplus D^*,
\]
where $E,D,D^*$ are $\rin$-bimodules. Let $\mathcal{R}$ be the one-dimensional $\Bbbk$-vector space spanned by 
\[
\sum_{a \in \Qu_1} [a^*,a] - \sum_{i \in \Qu_0} \lambda_i e_i. 
\]
Then the sub-$k$-bimodule of $E \otimes_k E$ generated by $\mathcal{R}$ is 
\[R:=
\Span_{\Bbbk} \left\langle \sum_{a \in \Qu_1, t(a) = i}\!\!\!\!\!\!\! a^* a -\!\!\!\! \sum_{a \in \Qu_1, h(a) = i}\!\!\!\!\!\!\! a a^* - \lambda_i e_i \ | \, i \in \Qu_0 \right\rangle
\]
and the (deformed) preprojective algebra is defined to be 
\[
\Pi^{\lambda}(\Qu) = T_\rin E / (R). 
\]
The undeformed preprojective algebra is $\N$-graded, with $\deg e_i = 0$ and $\deg a = \deg a^* = 1$. 

\begin{prop}
   Assume that $\Qu$ is not of finite ADE type. The (deformed) preprojective algebra is non-homogeneous Koszul and $\gr(\Pi^\lambda(\Qu))=\Pi^0(\Qu)$.
\end{prop}

\begin{proof}
    The quadratic algebra associated to the space of relations $R$ is the undeformed preprojective algebra $\Pi^0(\Qu)$. This is known to be Koszul, for $\Bbbk$ a field of any characteristic, when $\Qu$ is not of ADE type; see \cite{etingofKoszulityHilbertSeries2007} and the references therein. The fact that $\Pi^0(\Qu) \cong \gr(\Pi^\lambda(\Qu))$ is explained in \cite[Lemma~2,3]{crawley-boeveyDeformedPreprojectiveAlgebras2022}.
\end{proof}

In order to describe the Koszul dual of $\Pi^0(\Qu)$ we first consider $E^*$. It has a $\Bbbk$-basis $\{ \psi_a \, | \, a \in \overline{\Qu}_1 \}$, with $\psi_a(b) = e_{t(a)}$ if $b = a$ and $\psi_a(b) = 0$ otherwise. Recall that the $\rin$-bimodule structure on $E^*$ is
\[
(e_i \cdot \psi)(v) = e_i \psi(v), \quad (\psi \cdot e_j)(v) = \psi(e_j v), \quad \forall \, \psi \in E^*, v \in E, i, j \in \Qu_0. 
\]
In particular,
$$e_i \cdot \psi_a = \psi_{a e_i}, \quad \quad \psi_a \cdot e_i = \psi_{e_i a}, \quad \forall a\in\overline{\Qu}_1, \, i\in\Qu_0.$$ 
Recall formula \eqref{eq:dualtensorsemisimple} for the identification $E^*\ok E^*=(E\ok E)^*$:
\begin{align*}
    (\psi_a\ot\psi_b)(c\ot d)=\psi_b(\psi_a(c).d)=\psi_b(\delta_{a,c}e_{t(a)}d).
\end{align*}
So, if the path $ab$ is non-zero, $(\psi_a\ot\psi_b)(c\ot d)=\delta_{a,c}\delta_{b,d}e_{t(b)}$. For every non-zero path $\pi=a_n\ot\cdots\ot a_1$, we have an element $\psi_\pi$ of $(T_\rin E)^*$ given by $\psi_\pi:=\psi_{a_n}\ot\cdots\ot\psi_{a_1}$, such that $\psi_\pi(x)=e_{t(\pi)}$ if $\pi=x$, and $\psi_\pi(x)=0$ otherwise.

\begin{prop}\label{prop:Koszuldualpreproj}
    The Koszul dual $\Pi^0(\Qu)^!$ is the quotient of $T_\rin E^*$ by the quadratic relations:
    \begin{enumerate}
        \item[(R1)] $\psi_a \ot \psi_b, \psi_{a^*} \ot \psi_{b^*}$ for $a,b \in \Qu_1$;
        \item[(R2)] $\psi_a \ot \psi_{b^*},\psi_{a^*} \ot \psi_b$ for $a , b \in \Qu_1$, $a \neq b$;
        \item[(R3)] $\psi_a \ot \psi_{a^*} - \psi_{b} \ot \psi_{b^*}$ for $a,b \in \Qu_1$, $t(a) = t(b)$;
        \item[(R4)] $\psi_{a^*} \ot \psi_{a} - \psi_{b^*} \ot \psi_{b}$ for $a,b \in \Qu_1$, $h(a) = h(b)$;
        \item[(R5)] $\psi_a \ot \psi_{a^*} + \psi_{b^*} \ot \psi_{b}$ for $a,b \in \Qu_1$, $t(a) = h(b)$.
    \end{enumerate}
    As $\rin$-bimodules, $(A^!)^0 = \rin, (A^!)^1 = E^*_k, (A^!)^2 = \rin$ and $(A^!)^i = 0$ otherwise. 
\end{prop}

\begin{proof}
    It is easy to check that all of the above relations hold. Moreover, $(A^!)^0 = \rin, (A^!)^1 = E^*_k$ are immediate since the relations are all quadratic. If $\psi_{u} \ot \psi_{v} \ot \psi_{w}$ is a monomial of degree three, then, without loss of generality, we may assume that at least two of the three arrows $u,v,w$ belong to $\Qu_1$. Then applying (R2)--(R5), we may reorder $u,v,w$ so that $u,v \in \Qu_1$. Then (R1) implies that $\psi_{u} \ot \psi_{v} \ot \psi_{w} = 0$. We deduce that $(A^!)^i = 0$ for $i > 2$. 
    
    Assume that $\Qu$ is oriented such that there is at least one arrow $a(i)$ with $t(a(i)) = i$ for every $i \in \Qu_0$. Then it is clear from the relations that $A_2^!$ is spanned by all $\psi_{a(i)} \ot \psi_{a(i)^*}$ for $i \in \Qu_0$. Note that 
    \[
    e_j(\psi_{a} \ot \psi_{a^*} ) e_i = \psi_{ae_j} \ot \psi_{e_ia^*} = \psi_{a} \ot \psi_{a^*}
    \]
    if $i = j = t(a)$ and is zero otherwise. Therefore, it suffices to argue that $\dim_{\Bbbk} (A^!)^2 = |\Qu_0|$. By Proposition \ref{prop:koszuldualext}, it is sufficient to compute $\Ext^2_\Pi(k,k)$. We can compute $\Ext^{\idot}_{\Pi}(\rin,\rin)$ explicitly, using the standard resolution of the diagonal for $\Pi$ (see \cite[Theorem 2.7]{crawley-boeveyDeformedPreprojectiveAlgebras2022} and Remark \ref{rem:freeresolution}). We get
    $$ 0\rightarrow \Hom_\Pi(P_0,k)\rightarrow\Hom_\Pi(P_1,k)\rightarrow\Hom_\Pi(P_2,k)\rightarrow 0,$$
    where the maps are all zero, $P_0=P_2=\bigoplus_{i\in\Qu_0}\Pi e_i$, and $P_1=\bigoplus_{a\in\overline{\Qu}_1}\Pi e_{h(a)}$.
    It follows that 
    \[
    (A^!)^2 = \Hom_{\Pi}\left(\bigoplus_{i \in \Qu_0} \Pi e_i, \rin\right) \cong \rin.
    \]
    Here we have used that $\Qu$ is not of finite type. 
\end{proof}

As in the proof of Proposition~\ref{prop:Koszuldualpreproj}, we assume (without loss of generality) that $\Qu$ is oriented such that there is at least one arrow $a(i)$ with $t(a(i)) = i$ for every $i \in \Qu_0$. Then
\[
\omega_i:= \psi_{a(i)^*} \otimes \psi_{a(i)} \in E^*\ok E^*.
\]


\begin{prop}
    The Koszul dual $\Pi^\lambda(\Qu)^!$ is a cdga with differential equal to $0$ and curvature
    $$c=-\sum_{i\in\Qu_0}\lambda_i\omega_i.$$
\end{prop}
\begin{proof}
In this case, $\alpha=0$ and we claim that $\beta= -\sum_{i\in\Qu_0}\lambda_i\omega_i$. The $\rin$-bimodule $Q \subset E \otimes_{\rin} E$ has $\Bbbk$-basis spanned by 
\[
x_i := \sum_{a \in \Qu_1, t(a) = i} a^* \otimes a - \sum_{a \in \Qu_1, h(a) = i} a \otimes a^* \quad \textrm{for $i \in \Qu_0$}.
\]
The map $\beta \colon Q \to \rin$ sends $x_i$ to $-\lambda_i e_i$. Since $\omega_i(x_j) = \delta_{i,j} e_i$, we have $\beta = -\sum_{i\in\Qu_0}\lambda_i\omega_i$. Hence the claim follows from Theorem \ref{thm:cdga}.   
\end{proof}

\section{Derived Equivalences}\label{sec:equivalences}

A key feature of graded Koszul duality is an equivalence of derived categories. However, the usual construction of a derived category no longer works for curved dg-algebras. Positselski introduced a replacement known as the coderived category and proved that, for non-homogeneous Koszul algebras $U$, there exists an exact equivalence of triangulated categories in very great generality. Under the hypothesis that $A$ is of finite global dimension, this coderived category admits a simple description in terms of the homotopy category of injectives. 

In this section, we prove an equivalence between the derived category of $U$ and a Verdier localisation of the homotopy category of injectives of the Koszul dual. The construction is explicit, and we can completely characterize the class of objects killed by the Verdier localization. This works without any further assumption on $U$, except Koszulness. Under the hypothesis that $A$ has finite global dimension, we prove that the localisation is trivial, recovering the result by Positselski. Our proof is more similar in spirit to \cite{floystadKoszulDualityEquivalences2005}.

\subsection{The category of curved dg-modules}\label{sec:cdga}

Throughout this section, let $\Lambda = (\Lambda,d,c)$ denote a cdga with $\Lambda^0 = k$ and $\Lambda^i = 0$ for $i < 0$; we call such a cdga \textit{connected graded}. We say that the \coconnective (curved) dg-algebra $\Lambda$ is \textit{bounded} if, in addition, there exists $\ell \ge 0$ such that $\Lambda^i = 0$ for all $i > \ell$. Let $\grLmod{\Lambda}$ denote the abelian category of graded left $\Lambda$ modules. If $a \in \Lambda^i$ then we write $|a| = i$ for the degree of $a$. 

\begin{defn}
A curved $\Lambda$ module is a (cohomologically) graded left $\Lambda$-module $M = \bigoplus_{i \in \Z} M^i$ with a degree one $k$-linear map $d_M: M\to M$ such that
\begin{itemize}
    \item $d_M^2(m) = cm\quad $ for all $m \in M$.
    \item $d_M(am) = d(a)m + (-1)^{\lvert a \rvert} a d_M(m)\quad $ for all $ a\in \Lambda$ and $m \in M$.
\end{itemize}
A morphism $f: M \to N$ of curved $\Lambda$-modules is a degree 0 map of graded $\Lambda$-modules such that $fd_M = d_Nf$. Let $C(\Lambda)$ denote the category of curved $\Lambda$ modules. 
\end{defn}

\begin{rem}
    Let $M$ be a right $\Lambda$-module. There is an analogous notion of right cdg-module. The definition is the same, except that $d^2_M(m)=-mc$.
\end{rem}

Note that $C(\Lambda)$ is an Abelian category with degreewise kernels and cokernels. Let $\Lambda^\#$ denote the underlying graded algebra of $\Lambda$ and $M^\#$ the underlying graded $\Lambda^\#$ module of a curved $\Lambda$ module~$M$. 

\begin{defn}
A morphism $f: M \to N$ in $C(\Lambda)$ is null homotopic if there is a degree $-1$ map $s: M \to N$ of graded $\Lambda$-modules such that $f^i = d_N^{i-1}s^i + s^{i+1} d^i_M$ for all $i \in \mathbb{Z}$. The homotopy category $K(\Lambda)$ of $\Lambda$ is the quotient of $C(\Lambda)$ by the ideal of null-homotopic morphisms. 
\end{defn}

\begin{rem} \label{rem:dgenhancement}
    The standard mapping cone construction makes $K(\Lambda)$ a triangulated category. Note that $H^i(\HOM_\Lambda(M,N)) = \Hom_{K(\Lambda)}(M,N[i])$. In particular, $H^0(\HOM_\Lambda(M,N)) = \Hom_{K(\Lambda)}(M,N)$ and so this provides a dg-enhancement of $K(\Lambda)$. 
\end{rem}

We let $C(\Inj \Lambda)$ denote the full subcategory of $C(\Lambda)$ consisting of modules $M$ such that $M^\#$ is injective as a graded $\Lambda^\#$ module. Similarly we let $K(\Inj \Lambda)$ denote the full subcategory of $K(\Lambda)$ 
consisting of objects that are isomorphic, in $K(\Lambda)$, to such modules. One can easily show that $C(\Lambda)$ admits enough injectives in the sense that any $M \in C(\Lambda)$ can be embedded (as a curved module) in some $I \in C(\Inj \Lambda)$. 

\begin{lem}\label{lem:Injtriangulated}
    The category $K(\Inj \Lambda)$ is a triangulated thick subcategory of $K(\Lambda)$. 
\end{lem}

\begin{proof}
The category $K(\Inj \Lambda)$ is clearly closed under shifts and summands. Given a triangle 
\[
I \to M \to J \stackrel{+}{\longrightarrow}
\]
with $I$ and $J$ in $K(\Inj \Lambda)$, we can assume that $M$ is the mapping cone of a map $J[-1] \to I$. In this case, $M^\# \simeq I^{\#} \oplus J^{\#}$. 
\end{proof}

We will use the following socle-like construction. For $M\in C(\Lambda)$, let
\[
S(M) = \HOM_{\Lambda}(\rin,M) = \{ m \in M \,  | \,  am = 0 \text{ for all $a \in \Lambda$ with } \lvert a \rvert > 0 \} \subseteq M.
\]
Note that $S(M)$ is a curved $\Lambda$ submodule of $M$ with $d^2 |_{S(M)} = 0$.

\begin{lem}
    $S$ induces a functor $K(\Inj \Lambda) \to D(k)$, which we denote by the same letter.  
\end{lem}

\begin{proof}
We need to check that null homotopic maps are sent to zero. This follows from the isomorphisms $H^n (S(-)) = H^n(\underline{\Hom}_{\Lambda}(\rin,-)) = \Hom_{K(\Lambda)}(\rin, (-)[n])$. 
\end{proof}

\begin{remark}
    The category $K(\Inj \Lambda)$ can be very degenerate. For instance, if $\Lambda = \Bbbk[c] / (c^n)$ where $d = 0$, $c$ is the curvature and $n \ge 2$, then every injective object in $C(\Lambda)$ is of the form $\Lambda \oplus \Lambda[1]$, where $d(a,b) = (cb,a)$. Then Proposition~\ref{prop:injjacyclic} below shows that $I = 0$ in $K(\Inj \Lambda)$ for all $I$. This agrees with \cite[Proposition~3.2]{KLPvanishing} under the identification of $K(\Inj \Lambda)$ with the coderived category of $(\Lambda,d,c)$.  
\end{remark}

\subsection{The Adjoint Functors} 

In this section, we define the two functors $F,G$ and prove that they are an adjoint pair that descend to the level of homotopy categories. The results in this section already appeared in \cite{floystadKoszulDualityEquivalences2005} in the case $\rin$ is a field.

Let $U$ be a non-homogeneous Koszul $k$-algebra and $A$ the associated quadratic algebra, so that $A^!$ is a cdga by Theorem~\ref{thm:cdga}. 

\begin{defn}
Define the graded $U\!-\!A^!$ bimodule $T := U \otimes_k A^!$ with $T^i = U \otimes_k (A^!)^i$ and differential $d=d_\ot+1\ot d_{A^!}$, where
\[
d_\ot:U\ok A^!=U\ok k \ok A^!\xrightarrow{c_E}U\ok E \ok E^\ast \ok A^!\xrightarrow{m_U\ot m_{A^!}} U\ok A^!,
\]
$c_E$ is the coevaluation map $k \to E \ok E^\ast$ and $m_U$ and $m_{A^!}$ denote multiplication in $U$ and $A^!$ respectively. If $c_E(1) = \sum_{\alpha} x_{\alpha} \otimes \hat{x}_{\alpha}$ then 
\[
d_{\otimes}(u \otimes b) = \sum_{\alpha} u x_{\alpha} \otimes \hat{x}_{\alpha} b.
\]
\end{defn}

\begin{prop}\label{prop:diffsquareto0}
    The differential $d$ defined above makes $T$ into a {right} cdg-module for $A^!$.
\end{prop}

\begin{proof}

We only need to check that $d^2(u\ot a)=-u\ot ac$. By definition, 
$$d^2(u\ot a)=d^2_\ot(u\ot a)+d_\ot(u\ot d_{A^!}(a))+(1\ot d_{A^!})d_\ot(u\ot a)+u\ot d^2_{A^!}(a).$$
By Leibniz rule, $d_\ot(u\ot d_{A^!}(a))+(1\ot d_{A^!})d_\ot(u\ot a)=f$, where
$$f:U\ok A^!\xrightarrow{1\ot c_E\ot 1}U\ok E\ok E^*\ok A^!\xrightarrow{1\ot1\ot d_{A^!}\ot1}U\ok E\ok E^*\ok A^!\xrightarrow{m_u\ot m_{A^!}}U\ok A^!.$$
Since $d^2_{A^!}(a)=[c,a]$, it is sufficient to prove that $(d^2_\ot+f)u\ot a=-u\ot ca$. By associativity and the first part of Proposition \ref{prop:monoidalnonsense},  it is sufficient to show that
\begin{equation}\label{eq:squaredifferential}
    d^2_\ot:k\xrightarrow{c_{E\ot E}}E\ok E\ok E^*\ok E^*\xrightarrow{m_U\ok m_{A^!}}U\ok Q^*
\end{equation}
is equal to $1\mapsto -(1\ot d_{A^!})c_E(1)-1\ok c$.
Consider 
\begin{equation}\label{eq:inclusionprojection}
    Q\xhookrightarrow{i} E\ok E\xrightarrow{p} U_{\le1}\subset U, 
\end{equation}
which is equal to $-\alpha-\beta$. Notice that $m_U=p$ on $E\ok E$ and $m_{A^!}=i^*$ on $E^*\ok E^*$. 
So \eqref{eq:inclusionprojection} is the adjunct of \eqref{eq:squaredifferential} by Lemma \ref{lemma:adjunctioncoev}. Hence, \eqref{eq:squaredifferential} is equal to
$$k\xrightarrow{c_{Q}}Q\ok Q^*\xrightarrow{(-\alpha-\beta)\ot Q^*}(k\oplus E)\ok Q^*\subset U\ok Q^*. $$
By dualizing the second diagram in Proposition \ref{prop:monoidalnonsense}, this is equal to
$$k\xrightarrow{c_{U_{\le1}}}U_{\le1}\ok U_{\le1}^*\xrightarrow{U_{\le1}\ot(-\alpha^*-\beta^*)}U_{\le1}\ok Q^*\subset U\ok Q^*,$$
which is equal to $1\mapsto -(1\ot \alpha^*)c_E(1)-1\ok \beta^*(1)=-(1\ot d_{A^!})c_E(1)-1\ok c$.
\end{proof}

Denote by $C(U)$ the category of complexes of $U$-modules. For $M\in C(U)$ and $N\in C(A^!)$, define
\begin{equation}\label{eq:F-Gdef}
    F(N):=T\ot_{A^!}N\cong U\ok N, \quad G(M):=\HOM_U(T,M)\cong\HOM_{k}(A^!,M).
\end{equation}
Clearly, $F(N)$ and $G(M)$ are graded $U$ and $A^!$ modules respectively with canonical degree one endomorphisms:
\begin{equation}
    F(N)^n=\bigoplus_{i+j=n}T^i\ot_{A^!} N^j \cong U\ok N^n, \quad d(t\ot n)=d_T(t)\ot n+(-1)^{|t|}t\ot d_N(n),
\end{equation}
\begin{equation}        G(M)^n=\prod_{i\geq0}\Hom_U(T^i,M^{n+i}) \cong \prod_{i\geq0}\Hom_\rin((A^!)^i,M^{n+i}), \quad d(f)=d_Mf-(-1)^{|f|}fd_{T}.
\end{equation}

\begin{rem}\label{rem:otherdifferentialF}
    Identifying $F(N)=T\ot_{A^!}N\cong U\ok N$, the differential becomes 
    $$d(u\ok n)=\sum_\alpha ux_\alpha\ok\hat{x}_\alpha n+u\ok d_N(n).$$
\end{rem}

\begin{prop}
Formula \eqref{eq:F-Gdef} defines a pair of adjoint functors
\[
\begin{tikzcd}
F  \colon C(A^!) \arrow[r,shift left] & \arrow[l,shift left] C(U) \colon G ,
\end{tikzcd}
\]
that descends to the homotopy categories
\[
\begin{tikzcd}
F:  K(A^!) \arrow[r,shift left] & \arrow[l,shift left] K(U): G
\end{tikzcd}
\]
\end{prop}
    
\begin{proof}
We first need to check that the codomains are correct. Let $N\in C(A^!)$, $t\in T$, and $n\in N$. Then, by Proposition \ref{prop:diffsquareto0}, 
\begin{align*}
    d^2(t\ot n)&=d^2_T(t)\ot n+(-1)^{|t|+1}d_T(t)\ot d_N(n)+(-1)^{|t|}d_T(t)\ot d_N(n)+t\ot d^2_N\\
    &=-tc\ot n+t\ot cn=0, 
\end{align*}
so $F(N)\in C(U)$. Now let $M\in C(U)$, $f\in\HOM_U(T,M)$, and $t\in T$. Since the differential on $M$ squares to zero, 
\begin{align*}
    (d^2f)(t)&=-(-1)^{|f|+1}d_Mf(d_T(t))-(-1)^{|f|}d_Mf(d_T(t))-f(d^2_T(t))\\
    &=f(tc)=(c.f)(t),
\end{align*}
so $G(M)\in C(A^!)$. By the graded version \cite[Proposition 2.4.9]{nastasescuMethodsGradedRings2004} of the tensor-hom adjunction,
$$\HOM_U(F(N),M)=\HOM_{A^!}(N,G(M)). $$
Taking degree zero-cocycles on both sides, we get the first adjunction. Considering degree zero-cohomology on both sides, we get the adjunction at the level of homotopy categories.
\end{proof}

\subsection{Full-faithfulness}

In this section we show that the functor $G$ is fully-faithful (as a functor from the derived category of $U$), giving the first half of the proof of Theorem \ref{thm:mainequivalence}. A similar result was proved by Fl{\o}ystad in \cite{floystadKoszulDualityEquivalences2005} in the special case $c=0$. Our proof generalizes the argument given there.

Recall $U = \bigcup_{i \ge 0} \Fi_i U$ is a filtered algebra, with associated graded isomorphic to $A$ by \Cref{thm:PBW}. This filtration induces a filtration $F = \underset{{i \to \infty}}{\colim} \, F_i$ on the functor $F$, where 
\begin{equation}\label{eq:Ffiltration}
    F_i(N) = \cdots \to \Fi_{i+p} U \otimes_{\rin} N^p \to \Fi_{i+p+1} U\otimes_{\rin} N^{p+1} \to \cdots .
\end{equation}

\begin{lem}\label{lem:FGquasiM}
       Let $M$ be a $U$-module, thought of as an object of $C(U)$ concentrated in degree zero. Then the counit of the adjunction $FG(M) \to M$ is a quasi-isomorphism.
\end{lem}

\begin{proof}
    Assume first that $M \in C(U)$ is an arbitrary complex. Then,
\[
F_i G(M)^p = \Fi_{i + p} U \otimes_k G(M)^p=\Fi_{i + p} U \otimes_k \HOM_{U}(T,M)^p,
\]
where the differential on $F_iG(M)$ is given by
\[
d(u \otimes \phi) = \sum u x_{\alpha} \otimes  (\hat{x}_{\alpha}.\phi)+ u \otimes d_M \circ \phi - (-1)^{|\phi|} u\otimes \phi \circ d_{T}; 
\]
see Remark \ref{rem:otherdifferentialF}.

If $u \otimes \phi\in \Fi_{i + p} U \otimes_k \HOM_{U}(T,M)^p$, then the term $u \otimes d_M \circ \phi - (-1)^{|\phi|} u\otimes \phi \circ d_{T}$ belongs to 
\[
\Fi_{i+p} U \otimes_k \HOM_{U}(T,M)^{p+1} = \Fi_{(i-1)+(p+1)} U\ok \HOM_{U}(T,M)^{p+1}.
\]
Therefore, the differential on the associated graded
\[
(\gr F)G(M)=\bigoplus_{i\in \Z} \frac{F_i G(M)}{F_{i-1} G(M)} \cong A \otimes_k\HOM_{U}(T,M)
\]
equals $d(\overline{u} \otimes \phi) = \sum \overline{u} x_{\alpha} \otimes  (\hat{x}_{\alpha}. \phi)$. 

Now assume that $M$ is concentrated in degree zero. Since $A^!$ is bounded below, Proposition~\ref{prop:gradedkhom} implies that 
\[
(\gr F)G(M) \cong A \otimes_k\HOM_{U}(T,M)\cong A\ok\HOM_{k}(A^!,M) \cong A\ok {}^*(A^!)\ok M.
\]
Hence the associated graded complex $(\gr F)G(M)$ is isomorphic to the Koszul complex $\mc{K}(A)=A\ok {}^*(A^!)$ tensored on the right by $M$. Since $\rin$ is semisimple, $M$ is flat over $\rin$, so $\mc{K}\ok M$ is a projective resolution of $M$, thought of as a graded $A$ module via the quotient map $A \to A / A_{> 0} = k$. Thus, we have a quasi-isomorphism $\mc{K}\ok M\rightarrow M$. Moreover, the existence of the short exact sequence 
$$0\to F_{i-1}G(M)\to F_iG(M)\to \mc{K}_i\ok M\to0, $$
and the fact that $\mc{K} \otimes_{\rin} M$ is acyclic in all non-zero (internal) degrees imply that the inclusion $\iota \colon F_iG(M)\to F_{i+1}G(M)$ is a quasi-isomorphism for all $i\geq0$. 

Since the grading on ${}^{\ast} (A^!)$ is concentrated in non-positive degrees and $\Fi_{i} U = 0$ for $i < 0$, $F_i G(M) = 0$ for $ i < 0$ and $F_0 G(M)$ can be identified with $M$. Hence the inclusions $M = F_0 G(M) \stackrel{\iota}{\hookrightarrow} F_i G(M)$ are quasi-isomorphisms for all $i$. Therefore, using the fact that cohomology commutes with colimits \cite[Lemma~00DB]{stacks-project}, the morphism ${\iota}$ in  
\begin{equation}\label{eq:Fifiltr2}
M = F_0 G(M) \stackrel{\iota}{\longrightarrow} FG(M) = \underset{i \to \infty}{\mathrm{colim}} \, F_iG(M) \longrightarrow M
\end{equation}
is a quasi-isomorphism. Since the composite $M \to M$ in \eqref{eq:Fifiltr2} is the identity, we deduce that $FG(M) \to M$ is a quasi-isomorphism.
\end{proof}

\begin{thm}\label{thm:FGquasiK}
    For any $M \in C(U)$, the counit $FG(M) \to M$ is a quasi-isomorphism. 
\end{thm}

\begin{proof}

Step 1: $M = M^0$ is concentrated in degree $0$. This is precisely Lemma~\ref{lem:FGquasiM}. 

Step 2: assume that $M$ is a bounded complex. There are short exact sequences 
\[
0 \to \sigma_{> i} M \to M \to \sigma_{\leq i} M \to 0
\]
given by brutal truncations. Since $FG$ is exact, starting from Step~1 we can prove by induction that $FG(M) \to M$ is a quasi-isomorphism. 

Step 3: assume that $M$ is an arbitrary complex that is bounded above. Then $M = \underset{i \to - \infty}{\mathrm{colim}} \, \sigma_{\geq i} M$ and we have
\[
G(M) = {}^*(A^!) \otimes_k M = {}^*(A^!) \otimes_k \underset{i \to - \infty}{\mathrm{colim}}  \, \sigma_{\geq i} M = \underset{i \to - \infty}{\mathrm{colim}}  \, {}^*(A^!) \otimes_k \sigma_{\geq i} M = \underset{i \to - \infty}{\mathrm{colim}}  \, G(\sigma_{\geq i} M). 
\]
Here we have made the first and final identifications using Proposition~\ref{prop:gradedkhom}, since $M$ and all of the complexes $\sigma_{\geq i} M$ are bounded above, and the middle identification follows from the fact that tensor products commute with colimits \cite[{Lemma 00DD}]{stacks-project}. Since $F$ is a left adjoint it preserves colimits: 
\[
FG(M) = F \left( \underset{i \to - \infty}{\mathrm{colim}}  \, G(\sigma_{\geq i} M) \right) \iso \underset{i \to - \infty}{\mathrm{colim}}  \, F G(\sigma_{\geq i} M),
\]
and so the map $FG(M) \to M$ is a quasi-isomorphism. 

Step 4: the case of an arbitrary complex $M$. Fix some $p \in \mathbb{Z}$. Naturality and exactness of $F$ and $G$ induce a commutative diagram with exact rows
\[
\begin{tikzcd}
0 \arrow[r]  & FG(\sigma_{> p} M) \arrow[d] \arrow[r] & FG(M) \arrow[r] \arrow[d] & FG(\sigma_{\leq p} M) \arrow[r] \arrow[d] & 0  \\
0 \arrow[r]  & \sigma_{> p} M\arrow[r] & M  \arrow[r] & \sigma_{\leq p} M  \arrow[r] & 0.  
\end{tikzcd}
\]
For $i < p$, this induces a diagram  
\[
\begin{tikzcd}
H^{i} FG(\sigma_{> p} M) \arrow[r] \arrow[d] & H^{i} FG(M) \arrow[r] \arrow[d] & H^{i} FG(\sigma_{\leq p} M) \arrow[r] \arrow[d,"\wr"] & H^{i+1}FG(\sigma_{> p} M) \arrow[d]\\
0 \arrow[r] & H^{i}(M) \arrow[r,"\sim"] & H^{i}(\sigma_{\leq p} M) \arrow[r] & 0
\end{tikzcd}
\]
on cohomology so it remains to show that if $M = \sigma_{>p} M$, then $H^{i}(FG(M)) = H^{i+1}(FG(M)) = 0$ for all $i < p$. If $M$ is concentrated in a single degree $\ell > p$ then $H^{i+1}(F_jG(M)) = 0$ for $i < p$ since $H^{\idot}(F_j G(M)) = H^{\idot}(F_0 G(M)) \cong M$, as shown in the proof of Lemma~\ref{lem:FGquasiM}. By using the brutal truncations and inducting on the length of the complex, one can show that if $M = \sigma_{ > p} M$ is a bounded complex $H^{i+1}(F_jG(M)) = 0$ for $i < p$. 

Now for any $M = \sigma_{ > p} M$, we have $M = \underset{q \to \infty}{\mathrm{lim}} \, \sigma_{\leq q} M$. Since $G$ is a right adjoint it preserves limits. Since each $\Fi_j U$ is finite-dimensional, $\Fi_j U \otimes_k -$ commutes with limits. Hence
\begin{equation}\label{eq:FGlim}
    F_jG(M) = F_j G \left( \underset{q \to \infty}{\mathrm{lim}} \,  \sigma_{\leq q} M \right) = \underset{q \to \infty}{\mathrm{lim}} \,  F_j G(\sigma_{\leq q} M). 
\end{equation}
In order to use this equality to prove the vanishing $H^{i+1}(FG(M)) = 0$, we claim that for a fixed $j$ the system of complexes $F_j G(\sigma_{\leq q} M)$ satisfies the Mittag-Leffler condition. Indeed, since each $\sigma_{\leq q} M$ is bounded above, Proposition~\ref{prop:gradedkhom} implies that in cohomological degree $t$ the transition maps are 
\[
\begin{tikzcd}
    F_j G(\sigma_{\leq q} M)^t \ar[r] \ar[d,"\wr"] & F_j G(\sigma_{\leq q-1} M)^t \ar[d,"\wr"] \\
   \bigoplus_{r \in \Z} \Fi_{j + t-r} U \otimes_k {}^\ast (A^!)^{t-r} \otimes_k (\sigma_{\leq q} M)^r \ar[r] &  \bigoplus_{r \in \Z} \Fi_{j + t-r} U \otimes_k {}^\ast (A^!)^{t-r} \otimes_k (\sigma_{\leq q-1} M)^r.
\end{tikzcd}
\]

These are all epimorphisms because, for fixed $r \in \Z$, the map
\[
\Fi_{j + t-r} U \otimes_k {}^\ast (A^!)^{t-r} \otimes_k (\sigma_{\leq q} M)^r \to \Fi_{j + t-r} U \otimes_k {}^\ast (A^!)^{t-r} \otimes_k (\sigma_{\leq q-1} M)^r
\]
is an epimorphism since tensoring over $k$ is exact and the maps $(\sigma_{\leq q} M)^r \to (\sigma_{\leq q-1} M)^r$ are epimorphisms. Hence by \cite[Theorem~3.5.8]{Weibel} and equality \eqref{eq:FGlim}, there is a short exact sequence
\[
0 \to \underset{q \to \infty}{\mathrm{lim}^1} H^{i}(F_j G(\sigma_{\leq q} M)) \to H^{i+1}( F_jG(M)) \to \underset{q \to \infty}{\mathrm{lim}} H^{i+1}(  F_j G(\sigma_{\leq q} M)) \to 0.
\]
The outer two terms vanish for $i < p$ and so $H^{i+1}(F_jG(M)) = 0$ for $i < p$. Using the fact that cohomology commutes with colimits \cite[Lemma~00DB]{stacks-project}, we deduce that $H^{i+1}(FG(M)) = 0$ for $i < p$ as required. 
\end{proof}

\subsection{A quasi-isomorphism} 

In this section, and the next, we prove a quasi-isomorphism between $S(I)$ and $F(I)$ when $I$ is injective. We first consider the homogeneous case, then use a spectral sequence argument to extend it to the non-homogeneous setting.

Consider the associated graded $\gr{F}$ of the functor $F$. We have $\gr{F}(N)=A\ok N$, with grading and differential given by 
$$\gr_q{F}(N)^p=A_{p+q}\ok N^p, \quad d(a\ok n)=d_\ot(a\ok n)=\sum_\alpha ax_\alpha\ok \hat{x}_\alpha n. $$
Notice that the term involving the differential on $N$ dies in the associated graded. In particular, $\gr F$ only depends on the underlying graded $A^!$ module $N^\#$. Therefore, we define a functor $\mc{F}\colon \grLmod{A^!}\to C(\grLmod{A}) $, $\mc{F}(N) = A \ok N$, with the same differential but a different bigrading  
$$
\mc{F}(N)^p_q=\gr_{-q} F(N)^{p+q}=A_{p}\ok N^{p+q}. 
$$
Notice that since $A$ is non-negatively graded, $\mc F(N)$ is concentrated in non-negative cohomological degrees. The functor $\mc F$ is well-defined for any homogeneous Koszul algebra. In particular, we can replace $A$ by ${}^! A$ and consider the functor associated to the latter. Since $({}^!A)^!=A$, we obtain a functor ${}^! \mc{F} \colon \grLmod{A} \to C(\grLmod{{}^! A})$ into the category of chain complexes over graded ${}^! A$ modules given by ${}^! \mc{F}(M) = {}^! A \ot_{\rin} M$. The following technical lemma is a key step in the proof of the main result. 

\begin{lem}\label{lem:Koszuldualfunctor}
    There is an isomorphism of complexes of graded right $A$ modules
    \[
    {}^!\mc{F}(A)={}^! A \ot_{\rin} A\cong \HOM_{A}(\mc{K}(A), A), 
    \] 
    where $\mc{K}(A)$ is the Koszul complex of $A$. 
\end{lem}

\begin{proof}
First, let us clarify the gradings on both sides: 
\[
({}^! A \ot_{\rin} A)^p_q = ({}^! A)^p \ot_{\rin} A_{p+q}, \quad \HOM_A(\mc{K}(A),A)^p_q=\HOM_A(\mc{K}(A)^{-p},A)_q.
\]
Recall that $\mc{K}(A)^p_q=A_{p+q}\ok Q^{(p)}$, thus $\phi\in\HOM_A(\mc{K}(A),A)^p_q$ is a map $A_{-p+m}\ok Q^{(-p)}\to A_{q+m}$, for some $m$. By $A$-linearity,  $\phi$ is completely determined by an element in $\Hom_{k}(Q^{(-p)},A_{p+q})$.
We have a bigrading preserving isomorphism $\psi$ defined on the bigraded components as
\[
({}^! A \ot_{\rin} A)^p_q \cong {}^* (Q^{(-p)}) \ot_{\rin} A_{p+q}  \cong \HOM_{\rin}(Q^{(-p)}, A_{p+q})  \cong \HOM_A(\mc{K}(A),A)^p_q.
\]
Explicitly,
\[
\psi(f \otimes a)(1 \ot (q_1 \ot \cdots \ot q_p)) = f(q_1 \ot \cdots \ot q_p) a
\]
for $f \in ({}^! A)^p \cong {}^* (Q^{(-p)})$, $a \in A$ and $q_1 \ot \cdots \ot q_p \in Q^{(-p)}$. 

We check that the isomorphism of graded vector spaces is compatible with differentials. First, 
\begin{align*}
    \psi(d (f \ot a))(1 \ot (q_1 \ot \cdots \ot q_i)) & =  \psi \left(\sum_{\alpha} f \check{x}_{\alpha} \ot x_{\alpha} a \right) (1 \ot (q_1 \ot \cdots \ot q_i)) \\
    & = \sum_{\alpha} (f \check{x}_{\alpha})(q_1 \ot \cdots \ot q_i)  x_{\alpha} a .
\end{align*}
Under the identification ${}^* (E^{\ot (i-1)}) \ot_k {}^* E \cong {}^* (E \ot E^{\ot (i-1)})$ of \eqref{eq:dualtensorsemisimpleleft}, 
\[
(f \check{x}_{\alpha})(q_1 \ot \cdots \ot q_i) = \check{x}_{\alpha} (q_1 f(q_2 \ot \cdots \ot q_i))
\]
and hence $\sum_{\alpha} (f \check{x}_{\alpha})(q_1 \ot \cdots \ot q_i)  x_{\alpha} = q_1 f (q_2 \ot \cdots \ot q_i)$. So 
\[
\psi(d (f \ot a))(1 \ot (q_1 \ot \cdots \ot q_i)) = q_1 f (q_2 \ot \cdots \ot q_i) a. 
\]
On the other hand, if $\phi \in \HOM_{A}(A \ot_{\rin} Q^{(\idot)}, A)$ then 
\[
(d \phi)(1 \ot (q_1 \ot \cdots \ot q_i)) = \phi(d (1 \ot (q_1 \ot \cdots \ot q_i))) =  \phi(q_1 \ot (q_2 \ot \cdots \ot q_i)),
\]
and hence 
\begin{align*}
    d(\psi(f \ot a))(1 \ot (q_1 \ot \cdots \ot q_i)) & =  \psi(f \ot a)(q_1 \ot (q_2 \ot \cdots \ot q_i)) \\
    & =  q_1 \psi(f \ot a)(1 \ot (q_2 \ot \cdots \ot q_i)) \\
    & =  q_1 f (q_2 \ot \cdots \ot q_i) a. 
\end{align*}
This shows that $d(\psi(f \ot a)) = \psi(d (f \ot a))$. 
\end{proof}



\begin{rem}
    Notice that in Lemma~\ref{lem:Koszuldualfunctor}, the differential on $\HOM_A(\mc{K}(A),A)$ is not the one $\phi \mapsto d_A\circ \phi -(-1)^{|\phi|}\circ d_{\mc{K}} $ given by \eqref{eq:homdifferential} since the factor $-(-1)^{|\phi|}$ is missing. Rather, this differential is one usually used to compute $\EXT_A^{\idot}(\rin,A)$ as the cohomology of $\HOM_A(\mc{K}(A)^\idot,A)$ using the fact that the Koszul complex is a projective resolution of $\rin$.
\end{rem}

\begin{prop}\label{prop:calFext}
    For any graded left $A$-module $M$, 
    \[
    H^i({}^! \mc{F}(M)) \cong \underline{\Ext}^i_{A}(\rin,M)
    \]
    as graded $k$-modules.  
\end{prop}

\begin{proof}
Note that, by Lemma~\ref{lem:Koszuldualfunctor}, 
\begin{align*}
    {}^! \mc{F}(M) & = ({}^! A \ot_{\rin} A) \ot_A M \cong \HOM_{A}(\mc{K}(A), A) \ot_A M  \cong \HOM_{A}(\mc{K}(A), M).
\end{align*}
 Since $\mc{K}(A)$ is a graded projective resolution of $\rin$, the $i$-th cohomology of the complex $\HOM_{A}(\mc{K}(A), M)$ equals $\underline{\Ext}^i_{{A}}(\rin,M)$. 
\end{proof}

We will mainly use the following corollary of Proposition~\ref{prop:calFext}. Since $N$ is a graded $A^!$-module, it can be thought of as a complex of graded $A^!$-modules concentrated in cohomological degree zero. Then $S(N)$ is a complex of graded $\rin$-modules concentrated in homological degree zero. The rule $n \mapsto 1 \ot n$ defines a morphism of complexes $S(N) \to \mathcal{F}(N)$, since $d(1 \ot n) = 0$. 

\begin{cor}\label{cor:homokoszulcomplex}
    If $I$ is an injective graded $A^!$-module then the morphism $S(I) \to \mathcal{F}(I)$ is a quasi-isomorphism of complexes over $\rin$. In particular, $H^{\idot}(\mathcal{F}(I)) = H^0(\mc{F}(I)) = S(I)$. 
\end{cor}

\begin{proof}

By construction, $H^\idot(S(I))=H^0(S(I))=S(I)$. If $1\ok n\in \rin\ok I=\mc F(I)^0$ is a cocycle, then $n$ is in $S(I)$. In fact, for all $a\in E^*$, 

$$1\ok a n=\sum_\alpha 1\ok x_\alpha(a)\hat{x}_\alpha n=\sum_\alpha a(x_\alpha)\ok \hat{x}_\alpha n=(a\ok 1)d(1\ok n)=0. $$

Thus, it suffices to show that all the other cohomology groups are zero. Recall that $A = {}^! (A^!)$. Therefore, by Proposition~\ref{prop:calFext}, $H^i(\mc{F}(I)) \cong \underline{\Ext}^i_{A^!}(\rin,I)$ for all $i$. Since $I$ is graded injective, $\EXT^i_{A^!}(\rin,I)$ is zero for $i\neq0$ (see \cite[Corollary 2.4.8]{nastasescuMethodsGradedRings2004}).

\end{proof}

\subsection{A spectral sequence}\label{sec:spectral}

Let $I \in C(\Inj A^!)$. Recall from \eqref{eq:Ffiltration} that we have a filtration $\{ F_i(I) \}_{i \in \Z}$ on the complex $F(I)$ given by $F_i(I)^q = \Fi_{i+q} U \ok N^q$. In order to have a descending filtration, following the conventions in \cite[\href{https://stacks.math.columbia.edu/tag/012K}{Tag012K}]{stacks-project}, we define 
\[
F^i(I) := F_{-i} (I) =   \bigoplus_{q \in \Z} \Fi_{q-i} U \ot_k I^q \quad \textrm{with} \quad F^i(I)^p = \Fi_{p-i} U \ot_k I^p.
\]
Then $\gr^p F(I) = \gr_{-p} F(I)$ so that $\gr^p F(I)^q = A_{q-p} \ok I^q$. 

There is a spectral sequence associated to this filtration. The $0$th page is
\[
E_0^{p,q} = \gr^p F(I)^{p+q} = \mc{F}(I^{\#})^q_p. 
\]
The differential is 
\[
d_0^{p,q} \colon E_0^{p,q} = A_{q} \ot_{\rin} I^{p+q}  \to  A_{q+1} \ot_{\rin} I^{p+q+1} = E_0^{p,q+1}, \quad d_0^{p,q}(a \ot u) = \sum_{\alpha} a x_{\alpha} \ot \check{x}_{\alpha} u. 
\]
Then $E_1^{p,q} = H^{q}(\mc{F}(I^{\#})_p)$. Therefore, Corollary~\ref{cor:homokoszulcomplex} 
imply that $E_1^{p,q}= 0$ unless $q = 0$. When $q = 0$, $E_1^{p,0} = H^{0}(\mathcal{F}(I^{\#})_p) = S(I)^p$. The differential on $E_1^{p,q}$ is 
\[
d_1^{p,0} \colon E^{p,0}_1 = S(I)^p \to S(I)^{p+1} = E^{p+1,0}_1
\]
given by $d_I |_{S(I)}$, and $d_1^{p,q} = 0$ for $q \neq 0$. This implies that $E_2^{p,0} = H^p(S(I))$ and $E_2^{p,q} = 0$ otherwise. For $r \ge 2$, $d_r^{p,q} \colon E^{p,q}_r \to E^{p+r,-r+1 + q}_r$ must be zero for all $(p,q)$ because $E_r^{p,q} = 0$ for $q \neq 0$. Thus, $E_{\infty}^{p,0} = H^p(S(I))$ and $E_{\infty}^{p,q} = 0$ otherwise. 

We define a filtration on $S(I)$ as 
\[
S(I)_i = F^i(I) \cap S(I) = \bigoplus_{q \ge i} S(I)^q,
\]
so that $(S(I)^\idot,d_I |_{S(I)})$ becomes a complex of filtered $A^!$-modules. We get another spectral sequence $\{\mc{E}^{p,q}_r\}$ associated to it. Since $\gr_p S(I) = S(I)^{p}$, the first page is $\mc{E}_0^{p,q} = \gr_p S(I)^{p+q}$, which is equal to $S(I)^p$ for $q = 0$ and is zero otherwise. In particular, the differential $d_0\colon \mc{E}^{p,q+1}_0\to \mc{E}^{p,q+1}_0 $ is zero, and $\mc{E}_1^{p,q} = \mc{E}_0^{p,q}$. The differential on $\mc{E}_1^{\idot,0}=S(I)^\idot$ is just $d_I |_{S(I)}$. Thus, $\mc{E}_2^{p,0} = \mc{E}_{\infty}^{p,0} = H^p(S(I))$ and $\mc{E}_2^{p,q} = \mc{E}_{\infty}^{p,q} = 0$ otherwise. 

Now consider the morphism $\psi: S(I) \to F(I)$ given by $j \mapsto 1 \ot j$. With our choices of filtration, it is strictly filtered. On page zero, $\psi$ sends $\mc{E}^{p,0}_0=S(I)^p$ into $E^{p,0}_0=A_0 \ot_k I^p$ and is zero everywhere else. By Corollary \ref{cor:homokoszulcomplex}, $\psi$ is a quasi-isomorphism $\mc{E}^{\idot,q}_1 \to E_1^{\idot,q}$. Thus, we see that the map on page 2, and hence all subsequent pages, is an isomorphism (compatible with differentials). 

\begin{thm}\label{thm:quasiisoJF}
For any $I \in K(\Inj A^!)$, the morphism $S(I) \to F(I)$ is a quasi-isomorphism of complexes of $k$-modules. 
\end{thm}

\begin{proof}
    We have shown that the induced morphism on spectral sequences is eventually an isomorphism. Note that both of these spectral sequences collapse on the second page and hence are bounded. Thus, they are regular. Clearly, the two sequences converge to the same limit $H^{\idot}(S(I))$. 

    Next, we note that the filtration $F^{\idot}(I)$ is exhaustive, $F(I) = \bigcup_{i \in \Z} F_i(I)$, since any element of $U \otimes I$ can be written as a finite sum $\sum u_j \ot n_j$ with $n_j$ homogeneous. Similarly, since $S(I) = \bigoplus_p S(I)^p$, any element of $S(I)$ belongs to a finite sum of the $S(I)^p$, implying that the filtration $S(I)_{\idot}$ is exhaustive. We claim that both filtrations are also complete. Indeed, 
    \[
    \left( \lim_{\infty \leftarrow p} F(I) / F^p(I) \right)^n = \lim_{\infty \leftarrow p} F(I)^n / F^p(I)^n = \lim_{\infty \leftarrow p} U \ot_{\rin} I^n / U_{n-p} \ot_{\rin} I^n = U \ot_{\rin} I^n = F(I)^n.
    \]
    Similarly, 
    \[
    \left( \lim_{\infty \leftarrow p} S(I) / S(I)_p \right)^n = \lim_{\infty \leftarrow p} S(I)^n / S(I)^n_p = S(I)^n
    \]
    since $S(I)^n_p = S(I)^n$ if $n \ge p$ and $S(I)^n_p = 0$ when $n < p$. 

    Therefore, $H^{\idot}(S(I)) \to H^{\idot}(F(I))$ is an isomorphism by the Eilenberg-Moore Comparison Theorem \cite[Theorem~5.5.11]{Weibel}.
\end{proof}

\subsection{Proof of the main Theorem} \label{proofofMainThm}

In this section we complete the proof of Theorem~\ref{thm:mainequivalence}. We define $\mathcal{N}:=\{I\in K(\Inj A^!) \, | \, S(I)\ \text{is acyclic} \}$. Since $\mathcal{N}$ is the kernel of the exact functor $S$, it is a thick triangulated subcategory of $K(\Inj A^!)$ and we may form the Verdier quotient $K(\Inj A^!) / \mathcal{N}$, which is again a triangulated category.

\begin{lem}\label{lem:imageG}
The image of $G: K(U) \to K(A^!)$ lies in $K(\Inj A^!)$ and $G$ sends acyclic complexes in $K(U)$ to $\mc{N}$. 
\end{lem}

\begin{proof}
If $M \in K(U)$, then $G(M) = \HOM_{k}(A^!,M)$ so, by graded tensor-hom adjunction,
\[
\Hom_{\grLmod{A^!}}(-,\HOM_{\rin}(A^!,M)) = \Hom_{\grLmod{k}}(-,M).
\]
Since $k$ is semisimple, $M$ is an injective graded module and $\Hom_{\grLmod{k}}(-,M)$ is exact. Thus, $\Hom_{\grLmod{A^!}}(-,\HOM_{\rin}(A^!,M))$ is an exact functor on the category $\grLmod{A^!}$ of graded $A^!$-modules. This proves the first statement. Let $M \in K(U)$. Note that 
\begin{equation}
    H^{\idot}(S(G(M))) \cong H^{\idot}(M).
\end{equation} 
Indeed,  
\[
H^{i}(S(G(M))) = \Hom_{K(A^!)}(k,G(M)[i]) \simeq \Hom_{K(U)}(F(k),M[i]) = \Hom_{K(U)}(U,M[i]) = H^{i}(M). 
\]
Therefore, if $M$ is an acyclic complex then so too is $S(G(M))$. By the previous paragraph, $G(M)$ belongs to $K(\Inj A^!)$. Thus, $G(M) \in \mc{N}$ if $M$ is acyclic. 

\end{proof}

\begin{prop}\label{prop:FGadjoint}
The functors descend to an adjoint pair
\[
\begin{tikzcd}
F:  K(\Inj A^!) / \mc{N} \arrow[r,shift left] & \arrow[l,shift left] D(U) \colon G.
\end{tikzcd}
\]
\end{prop}

\begin{proof}
Let $\mr{Acyc} \subset K(U)$ be the full subcategory of acyclic complexes. Lemma~\ref{lem:imageG} says that the functor $G \colon K(U) \to K(\Inj A^!) \to K(\Inj A^!) / \mc{N} $ sends acyclic complexes to zero. Therefore, it factors uniquely through $D(U) = K(U) / \mr{Acyc}$, the Verdier localisation at acyclic complexes $\mr{Acyc}$. Next, if $N \in \mc{N}$ then Theorem~\ref{thm:quasiisoJF} implies that $F(N)$ is acyclic. Therefore, the functor $F \colon K(\Inj A^!) \to K(U) \to D(U)$ also factors through the Verdier localisation $K(\Inj A^!) / \mc{N}$. 

To see that the adjunction holds, we note that Theorem~\ref{thm:quasiisoJF} implies that 
\begin{equation}\label{eq:kernelFisN}
   \mc{N} = \{ I \in K(\Inj A^!) \, | \, F(I) \in \mr{Acyc} \}. 
\end{equation}
Therefore, adjunction follows from the abstract result \cite[Lemma 1.1.6]{KrauseBook}. 

\end{proof}

Equality \eqref{eq:kernelFisN} implies:

\begin{prop} \label{prop:Fconservative}
The functor $F \colon K(\Inj A^!) / \mc{N} \to D(U)$ is conservative. 
\end{prop}

Next, we consider the functor $G$. 

\begin{prop}\label{prop:fullyfaithfullGundboundedquot}
The functor $G \colon D(U) \to K(\Inj A^!) / \mc{N}$ is fully faithful.
\end{prop}

\begin{proof}
    Proposition~\ref{prop:FGadjoint} and Theorem~\ref{thm:FGquasiK} together imply that the counit $FG \to 1$ is an isomorphism in $D(U)$, which implies that $G$ is fully faithful. 
\end{proof}

\begin{thm}\label{thm:main}
The functors $F$ and $G$ induce inverse triangulated equivalence $D(U) \simeq K(\Inj A^!)/\mathcal{N}$.

\end{thm}

\begin{proof}
By \Cref{prop:FGadjoint}, \Cref{prop:Fconservative} and \Cref{prop:fullyfaithfullGundboundedquot}, $F$ is a conservative functor with a fully faithful right adjoint. It follows that $F$ is an equivalence, by e.g. \cite[Proposition~3.2.9]{KrauseBook}. 
\end{proof}

\begin{example}\label{ex:dualnumbers}
    If $U = \Bbbk[\epsilon] / (\epsilon^2)$ so that $A^! = \Bbbk[x]$ then $\mathcal{N}$ contains objects other than $0$. Indeed, if $I = \Bbbk[x,x^{-1}]$ as in \Cref{rem:examplebad} then $I$ is non-zero as an object of $K(\Inj A^!)$ but $F(I)$ is the acyclic complex $\cdots \to U \stackrel{\epsilon \cdot -}{\rightarrow} U \to \cdots $. 
\end{example}

\subsection{Locally nilpotent modules}\label{sec:locallynilp}

In this section, we study the category of ``locally nilpotent cdg-modules'' and show that its intersection with $\mc N$ is zero. This will be used in the next two sections, where we refine Theorem \ref{thm:main}. In \cite{positselskiRelativeNonhomogeneousKoszul2021}, these modules are known as comodules.

As in Section~\ref{sec:cdga}, let $(\Lambda,d,c)$ be a \coconnective curved dga. A $\Lambda$-module $M$ is said to be \textit{locally nilpotent} if for each $m \in M$ there exists $\ell > 0$ such that $\Lambda_+^\ell \cdot m= 0$. The full subcategory of $C(\Lambda)$ consisting of cdg-modules whose underlying module is nilpotent is denoted $C(\Lambda)_{\mathrm{nil}}$. If the graded algebra $\Lambda$ is bounded then $C(\Lambda)_{\mathrm{nil}} = C(\Lambda)$. If $M$ is locally nilpotent then $M = 0$ in $C(\Lambda)$ if and only if $S(M) = 0$. Moreover, for $M$ locally nilpotent, every element $m\in M$ is contained in a bounded submodule $\Lambda.m$. In particular, $M$ is a countable union of bounded submodules.

Let $K(\Inj \Lambda)_{\nil}$ denote the full subcategory of $K(\Inj \Lambda)$ consisting of objects isomorphic to the image of some $I \in C(\Lambda)_{\mathrm{nil}}$. Then, just as in the proof of Lemma~\ref{lem:Injtriangulated}, one can check that $K(\Inj \Lambda)_{\nil}$ is closed under summands, shifts, and triangles and hence is a thick subcategory of $K(\Inj \Lambda)$.

\begin{prop}\label{prop:injjacyclic}
Let $I \in C(\Inj \Lambda)_{\nil}$. Then $I = 0$ in $K(\Inj \Lambda)_{\nil}$ if and only if $S(I)$ is acyclic. In particular, $K(\Inj \Lambda)_{\nil}\cap \mc N=0$.
\end{prop}

\begin{proof}
Let $I \in C(\Inj \Lambda)_{\nil}$ and assume that $S(I)$ is acyclic. We define $C(\Lambda)_I$ to be the full subcategory of $C(\Lambda)_{\nil}$ consisting of modules $M$ such that $\HOM_{\Lambda}(M,I)$ is acyclic. By assumption, $\rin \in C(\Lambda)_I$. We will show that $C(\Lambda)_I = C(\Lambda)_{\nil}$. Note that $C(\Lambda)_I$ is closed under shifts and summands. We claim that it is also closed under arbitrary coproducts. Indeed, the category of locally nilpotent graded modules is closed under arbitrary coproducts. Now, let $\{ M(j) \}_{j \in J}$ be a collection of objects in $C(\Lambda)_I$. Then, for each $r \in \Z$, 
\[
   \HOM_{\Lambda}\Big( \bigoplus_{j \in J} M(j), I\Big)^r  = \Hom_{\grLmod{\Lambda}} \Big( \big(\bigoplus_{j \in J} M(j) \big), I[r] \Big) = \prod_{j \in J} \Hom_{\grLmod{\Lambda}} ( M(j), I[r]),
   \]
   which shows that $\HOM_{\Lambda}\big( \bigoplus_{j \in J} M(j), I\big)$ is the term-wise product of the acyclic complexes $\HOM_{\Lambda}(M(j), I)$. This is acyclic, see \cite[Exercise~1.2.1]{Weibel}. 

Assume next that there is a short exact sequence 
\[
0 \to N \to M \to L \to 0 
\]
in $C(\Lambda)$ with $N,L \in C(\Lambda)_I$. Since $I^\#$ is injective, applying $\HOM_{\Lambda}(- ,I)$ gives a short exact sequence of complexes and hence a long exact sequence in cohomology. This forces $M \in C(\Lambda)_I$ too i.e. $C(\Lambda)_I$ is closed under (3) extensions. We will show that it is also closed under (4) countable unions. By this we mean that if $M \in C(\Lambda)$ with $M^{(1)} \subset M^{(2)} \subset \cdots \subset M$, $M = \bigcup_{i \ge 1} M^{(i)}$ and $M^{(i)} \in C(\Lambda)_I$ for all $i$ then $M \in C(\Lambda)_I$. Consider the diagram of complexes
\[
\dots \to \HOM_\Lambda(M^{(1)},I) \to \HOM_\Lambda(M^{(0)},I) 
\]
Its limit is $\HOM_\Lambda(M,I)$ since representable functors commute with colimits. Each of the maps is epi since $I$ is injective. This implies that it satisfies the Mittag-Leffler condition \cite[(3.5.6)]{Weibel}. Therefore, by \cite[Theorem~3.5.8]{Weibel}, the vanishing of $H^\idot(\HOM_\Lambda(M^{(i)},I))$ implies the vanishing $H^\idot(\HOM_\Lambda(M,I))$. Therefore $\HOM_\Lambda(M,I)$ is acyclic. 

Next, assume that $M \in C(\Lambda)$ is concentrated in a single degree. This means that the action of $\Lambda$ on $M$ factors through $\rin$. Since the latter is semisimple, $M$ is a shift of a sum of summands of copies of $\rin$ and so $M \in C(\Lambda)_I$.  Next suppose $M_i = 0$ for all $|i| > N$ then we again claim that $M \in C(\Lambda)_I$. Indeed, if $|\Supp \, M | = n$ (here $\Supp \, M = \{ i \in \Z | M_i \neq 0\}$) and $M_{\ell} \neq 0$ but $M_{j} = 0$ for all $j > {\ell}$ then we have a short exact sequence
\[
0 \to M_{\ell} \to M \to N \to 0
\]
in $C(\Lambda)$ since $\Lambda$ is connected graded. Since $|\Supp N | \le n-1$ we have $N \in C(\Lambda)_I$ by induction. Hence so too is $M$.  

Finally, if $M$ is arbitrary, let $M^{(n)}$ be the largest cdg-submodule (equivalently, the sum of all cdg-submodules) whose support $\Supp \, M^{(n)} \subset [-n,n]$. Since $M$ is assumed to be locally nilpotent, $M = \bigcup_{n \in \N} M^{(n)}$. Since each $M^{(n)} \in C(\Lambda)_I$ and $C(\Lambda)_I$ is closed under countable unions we conclude that $M \in C(\Lambda)_I$. This proves the claim that $C(\Lambda)_I = C(\Lambda)_{\nil}$. In particular, $I \in C(\Lambda)_I$.  

Then, since $H^0(\HOM_\Lambda(M,I)) = \Hom_{K(\Lambda)}(M,I)$, we have $\Hom_{K(\Lambda)}(I,I) = 0$. This implies that $I \simeq 0 \in K(\Lambda)$. 
\end{proof}

Proposition~\ref{prop:injjacyclic} implies that the functor $S \colon K(\Inj \Lambda) \to D(k)$ is conservative on $K(\Inj \Lambda)_{\nil}$. 

\begin{remark}\label{rem:examplebad}
    If $\Lambda$ is a \coconnective cdga that is not bounded, then Proposition~\ref{prop:injjacyclic} fails for modules $I \in K(\Inj \Lambda)$ that are not locally nilpotent. For instance, let $\rin = \Bbbk$ be a field and $\Lambda = k[x]$, where $x$ is in degree one, with trivial differential and zero curvature. Then $I = k[x,x^{-1}]$ is injective in the category of graded $\Lambda$-modules \cite[1.2.4~Lemma]{GradedRingTheory} and can be thought of as a curved module with zero differential. Then $\underline{\End}_{\Lambda}(I,I)\cong k[t,t^{-1}]$ as graded module, with trivial differential. Hence $\End_{K(\Lambda)}(I) = H^0(k[t,t^{-1}])= k$. But $\underline{\Hom}_{\Lambda}(k,I) = 0$. 
\end{remark} 

\subsection{Complexes bounded above}
We return to the setting of filtered Koszul duality. 

Let $K^-(U)$ (resp. $K^-(\Lambda)$) be the full subcategory of $K(U)$ (resp. of $K(A^!)$) consisting of all complexes (resp. $A^!$-modules) $M$ such that $(M^{\#})_i = 0$ for all $i \ge N$ and some $N \in \Z$. 

\begin{lem}\label{lem:imageG2Noeth}
 The functor $G$ sends acyclic complexes in $K^-(U)$ to zero. 
\end{lem}

\begin{proof}
   Let $M \in K^-(U)$. By Proposition~\ref{prop:gradedkhom}, $G(M)$ is bounded above, so it belongs to $K^-(\Inj A^!)$. But Lemma~\ref{lem:imageG} says that $G(M)$ also belongs to $\mc{N}$. Therefore, since bounded above modules are locally nilpotent, it suffices to note that Proposition~\ref{prop:injjacyclic} says that the intersection $K(A^!)_{\nil} \cap \mc{N}$ is zero. 
\end{proof}

\begin{lem} \label{lem:FconservativeNoethbound}
The functors $(F,G)$ descend to an adjoint pair
\[
\begin{tikzcd}
F:  K^-(\Inj A^!) \arrow[r,shift left] & \arrow[l,shift left] D^-(U) \, \colon G.
\end{tikzcd}
\]
Moreover, the functor $F: K(\Inj A^!)_{\nil} \to D(U)$ is conservative and the functor $G \colon D^-(U) \to K^-(\Inj A^!)$ is fully faithful. 
\end{lem}
\begin{proof}
That the pair is adjoint follows by repeating the proof of Proposition~\ref{prop:FGadjoint}, but with Lemma~\ref{lem:imageG2Noeth} replacing Lemma~\ref{lem:imageG}. The functor $F$ is conservative because of Proposition \ref{prop:injjacyclic} and Theorem \ref{thm:quasiisoJF}. The functor $G$ is fully faithful on all of $D(U)$, so it is fully faithful on a subcategory.
\end{proof}

Since Lemma~\ref{lem:FconservativeNoethbound} says that $F$ is a conservative functor with a fully faithful right adjoint $G$, it follows that $F$ is an equivalence, by e.g., \cite[Proposition~3.2.9]{KrauseBook}. 

\begin{cor}\label{cor:main2abound}
     The functors $F$ and $G$ induce inverse equivalence $D^-(U) \simeq K^-(\Inj A^!)$.
\end{cor}

\subsection{Finite global dimension}

In this section, we show that if $A$ has finite global dimension then $\mc{N} = 0$ and hence the main equivalence of Theorem~\ref{thm:mainequivalence}  is an equivalence $D(U) \simeq K(\Inj A^!)$.

\begin{lem}\label{lem:finglbdim}
    The algebra $A$ has (left or right) global dimension $n < \infty$ if and only if $(A^!)^n \neq 0$ but $(A^!)^i = 0$ for all $i > n$. 
\end{lem}

\begin{proof}
If $A$ is graded left Noetherian, then it is well-known that it has finite global dimension if and only if $k$ has finite projective dimension as a graded $A$-module since $A$ is a connected graded $k$-algebra; see \cite{LiGlobalDimGraded}. However, in our setting $A$ need not be Noetherian. Therefore, we give a direct proof using the functors $(F,G)$. 

Notice that if $M$ is a graded left $A$-module and $A^!$ is bounded above at degree $n$, then the quasi-isomorphism $FG(M) \to M$ of Lemma~\ref{lem:FGquasiM} is a projective resolution of $M$ of length $n$ and hence $A$ has (left) global dimension at most $n$. On the other hand, Proposition \ref{prop:koszuldualext} implies that $\Ext^p_{A}(k,k) = ({}^! A)^p$,
so the left global dimension of $A$ is exactly $n$.

The case of right global dimension is dealt with by considering ${}^! A$ instead and noting that $(A^!)^p \neq 0$ if and only if $({}^! A)^p \neq 0$. 
\end{proof}

\begin{lem}\label{lem:imageG2}
If $A$ has finite global dimension then $G$ maps acyclic complexes in $K(U)$ to zero. 
\end{lem}

\begin{proof}
If $A$ has finite global dimension then $A^!$ is bounded by Lemma~\ref{lem:finglbdim} and thus finite-dimensional. In particular, every curved $A^!$ module is locally nilpotent and Proposition~\ref{prop:injjacyclic} says that $\mc{N} = \{ 0 \}$. Therefore, the lemma follows directly from Lemma~\ref{lem:imageG}. 
\end{proof}

Repeating the proof of Lemma~\ref{lem:FconservativeNoethbound}, but using Lemma~\ref{lem:imageG2} instead of \ref{lem:imageG2Noeth}, implies that $(F,G)$ descend to an adjoint pair 
\[
\begin{tikzcd}
F:  K(\Inj A^!) \arrow[r,shift left] & \arrow[l,shift left] D(U) \colon G.
\end{tikzcd}
\]
Moreover, $F$ is a conservative functor with a fully faithful right adjoint $G$, hence:

\begin{cor}\label{cor:main2b}
If $A$ has finite global dimension then $F, G$ extend to equivalences $D(U) \simeq K(\Inj A^!)$.
\end{cor}

\subsection{Relation to work of Positselski}\label{sec:Positselski}

We review our results in the context of Positselski's work. One of the main results, as stated in \cite{positselskiTwoKindsDerived2011, positselskiDifferentialGradedKoszul2023}, gives a duality between the (ordinary) derived category of a dg-algebra and the coderived category of a cdg-coalgebra. There is no assumption of Koszulness, which is only needed to provide the ``small'' description of the dual. The general theory in these papers is developed for algebras over a field $\Bbbk$. In \cite{positselskiRelativeNonhomogeneousKoszul2021}, many of these results, including the special case of non-homogeneous Koszul duality, are extended to algebras over a much more general class of rings, which trivially includes semisimple rings. Let us unpack some of the definitions.

The notion of cdg-coalgebra is dual to that of cdg-algebra. That is, in the definition all arrows are inverted. Similarly, one can define cdg-comodules dual to cdg-modules. Recall that for a coalgebra $C$, the vector space dual $C^*$ is naturally an algebra, but the converse is true only for finite-dimensional algebras. Similarly, the graded dual of a cdg-coalgebra is a cdg-algebra, but the converse is true only if all graded components are finite-dimensional. 

Given a cdg-coalgebra $C$ and a dg-algebra $A$, the space $\HOM_\Bbbk(C,A)$ has a canonical structure of a cdg-algebra. A \emph{twisting cochain} $\tau$ for $C$ and $A$ is a Maurer-Cartan element for the cdg-algebra $\HOM_\Bbbk(C,A)$; that is a degree $1$ linear map $\tau\colon C\to A$ satisfying the Maurer-Cartan equation (see \cite[Section 7.8]{positselskiDifferentialGradedKoszul2023}). The coderived category of cdg-comodules of $C$ is the Verdier quotient of the homotopy category by the subcategory of \emph{coacyclic cdg-comodules} (see \cite[Definition 7.11]{positselskiDifferentialGradedKoszul2023}). There is an equivalence 
$$ K(\Inj C) \simeq D^{co}(C), $$
where $K(\Inj C)$ is the homotopy category of cdg-comodules whose underlying graded comodule is injective \cite[Theorem 4.4(c)]{positselskiTwoKindsDerived2011}.

\begin{thm}[{\cite[Theorem 6.12]{positselskiDifferentialGradedKoszul2023}
}]
    Let $A$ be a nonzero dg-algebra and $C$ a conilpotent cdg-coalgebra, with coaugmentation $\gamma$. Let $\tau\colon C\to A$ be an acyclic twisting cochain such that $\tau\circ\gamma=0$. Then there is an equivalence of triangulated categories:
    $$D(A)\simeq D^{co}(C).$$
\end{thm}

Through the bar construction, one can obtain a coalgebra $\mr{Bar^\bullet}(A)$ with a twisting cochain starting from any algebra $A$ \cite[Example 6.10]{positselskiDifferentialGradedKoszul2023}. This coalgebra is in general rather big; to get smaller examples Koszulness is needed. Let $U$ be a non-homogeneous Koszul algebra and $A=\gr \, U$. Consider the cdg-coalgebra $A^?$, the graded dual of $A^!$. The coalgebra $A^?$ is a cdg-subcoalgebra of $\mr{Bar^\bullet}(A)$ and there is a canonical twisting cochain $\tau\colon A^? \to U$ \cite[Example 6.11]{positselskiDifferentialGradedKoszul2023} that gives an equivalence 
\begin{equation}\label{eq:coalgebraequivalence}
    D(U)\simeq D^{co}(A^?)\simeq K(\Inj A^?). 
\end{equation}

We wish to translate the equivalence \eqref{eq:coalgebraequivalence} to an equivalence between $D(U)$ and some category of \emph{modules} for $A^!$. The issue is that not all $A^!$-modules can be obtained by dualizing $A^?$-comodules; only the locally nilpotent modules are obtained this way. The coderived category $D^{co}(A^!)$ is then defined as a quotient of $K(A^!)_\nil$. It is equivalent to the homotopy category $K(\Injn A^!)$, which is the subcategory of $K(A^!)_\nil$ consisting of all objects isomorphic to a nilpotent cdg-module whose underlying graded module is injective \emph{in the category of locally nilpotent graded modules}.

For the rest of the section, let $U$ be a non-homogeneous Koszul algebra over a semisimple ring $\rin$ and $A=\gr \, U$.

\begin{thm}\label{thm:equivalenceinjnilp}
    There is an equivalence of triangulated categories 
     \begin{equation}
    D(U)\simeq K(\Injn A^!). 
\end{equation}
\end{thm}
This follows from \cite[Corollary 6.18, Theorem 8.17]{positselskiRelativeNonhomogeneousKoszul2021}. We provide a direct proof, based on our previous results, for the reader's convenience. We need some preliminary results. Let $\mb F$ be the restriction of $F$ to $C(A^!)_\nil$. Define a new functor $\mb G\colon C(U)\to C(A^!)$ by
\begin{equation}
        \mb{G}(M)=\bigoplus_{n\in\Z}G(M)^n, \quad \mb G(M)^n:=\bigoplus_{i\geq0}\Hom_U(T^i, M^{n+i})=\bigoplus_{i\geq0}\Hom_\rin((A^!)^i, M^{n+i}).
\end{equation}
That is, $\mb G$ is the direct sum totalization of the bi-complex of graded homomorphisms (whilst the original functor $G$ was the direct product totalization). 

Let $\Gamma\colon C(A^!)\to C(A^!)_\nil$ be the functor that sends a module to its largest locally nilpotent submodule. 

\begin{lem}\label{lem:tildeG}
    Let $N\in C(A^!)$ and $M\in C(U)$. Then $S(N)=S(\Gamma(N))$ and $\Gamma(G(M))=\mb G(M)$.
\end{lem}
\begin{proof}
    Let $N'\in C(A^!)_\nil$, $N\in C(A^!)$. For every $f\in\HOM_{A^!}(N',N)$, the image of $f$ is a locally nilpotent submodule of $N$ and hence contained in $\Gamma (N)$. Therefore, $\HOM_{A^!}(N',\Gamma (N))=\HOM_{A^!}(N',N)$. In particular, $S(N)=\HOM_{A^!}(\rin,N)=S(\Gamma(N))$. 

    Let $f\in G(M)$ be a nilpotent element. We can assume $f$ is homogeneous, say $f\in G(M)^n$. Then $f=(f_i)_{i\geq0}$, with $f_i\in \Hom_\rin((A^!)^i,M^{n+i})$. Clearly, for all $i\geq0$, if $f_i\neq 0$ then $(A^!)^{\leq i} . f_i \neq 0$. So, for all but finitely many $i$, $f_i=0$. Hence $f\in \mb G(M)$, which implies that $\Gamma (G(M))=\mb G(M)$.   
\end{proof}

\begin{proof}[Proof (Theorem \ref{thm:equivalenceinjnilp})]
    First, notice that $\mb F, \mb G$ are still an adjoint pair
    \begin{equation*}
        \begin{tikzcd}
\mb F:  K(\Injn A^!) \arrow[r,shift left] & \arrow[l,shift left] K(U) \, \colon \mb G.
\end{tikzcd}
    \end{equation*}
    In fact, by Lemma \ref{lem:tildeG}, for all $N,N'\in C(A^!)_\nil$:
    \begin{equation*}
        \HOM_{\grLmod{A^!}}(N',\mb G(M))=\HOM_{\grLmod{A^!}}(N',G(M)),
    \end{equation*}
    which is exact because $G(M)$ is injective. Moreover, for all $M\in C(U)$: 
    \begin{equation*}
        \HOM_{U}(F(N),M)=\HOM_{A^!}(N,G(M))=\HOM_{A^!}(N,\mb G(M)).
    \end{equation*}   
    Notice that the proof of Corollary \ref{cor:homokoszulcomplex} uses the injectivity of $I$ only to show that $\EXT_{A^!}^i(\rin, I)=0$ for $i\neq0$. Since $\rin$ is a nilpotent module, this vanishing holds for $I$ injective in $C(A^!)_\nil$. In particular, Corollary \ref{cor:homokoszulcomplex} holds for every $I\in C(\Injn A^!)$. The same is true for Theorem \ref{thm:quasiisoJF}, because the spectral sequence argument does not use that $I$ is injective. 
    
    Let $M\in C(U)$. We have quasi-isomorphisms
    $$ F(G(M))\simeq S(G(M)) = S(\mb G(M))\simeq \mb F(\mb G(M)).   $$
    Therefore, Theorem \ref{thm:FGquasiK} implies that $\mb F\mb G(M)\simeq M$. So, if $M$ is acyclic, so is $\mb F \mb G(M)$ and so is $S(\mb G(M))$, which means that $\mb G(M)\in \mc N$. The proof of Proposition \ref{prop:injjacyclic} only uses the fact that $I$ is injective in $K(A^!)_\nil$, so $K(\Injn A^!)\cap \mc N=0$. Hence,  $\mb G(M)$ must be zero. This proves that $\mb G$ factors through a functor 
    $\mb G\colon D(U)\to K(\Injn A^!)$. 
    The proof then follows by the same argument of Theorem \ref{thm:main}: we have an adjoint pair
    \begin{equation*}
        \begin{tikzcd}
\mb F:  K(\Injn A^!) \arrow[r,shift left] & \arrow[l,shift left] D(U) \, \colon \mb G.
\end{tikzcd}
    \end{equation*}
    where $\mb F$ is conservative and $\mb G$ is fully-faithful.     
\end{proof}

\begin{lem}\label{lem:gradednoetherian}
    Let $\Lambda$ be a connected graded left Noetherian $\rin$-algebra. A locally nilpotent graded $\Lambda$-module is an injective graded module if and only if it is an injective object in the category of locally nilpotent graded modules.
\end{lem}
\begin{proof}
If $I$ is a locally nilpotent injective graded $\Lambda$-module, then it is clearly injective as an object in the category of locally nilpotent graded modules. We show the converse under the (graded left) Noetherian hypothesis.

Let $I$ be a locally nilpotent graded $\Lambda$-module that is injective as an object in the category of locally nilpotent graded modules. Notice that $S(I) \neq 0$ since $I$ is locally nilpotent. The graded $\Lambda$ module 
    \[
    J = \bigoplus_{p \in \Z} \HOM_{k}(\Lambda,S(I)^p)
    \]
    is a direct sum of injective graded left $\Lambda$ modules. Since we have assumed that $\Lambda$ is graded left Noetherian, this implies that $J$ is injective. Then the injectivity of $J$ implies that the embedding $S(I) \to J$ extends to a morphism $i \colon I \to J$. Let $N$ be the kernel of this morphism. Since $I$ is locally nilpotent, $N \neq 0$ implies that $S(N) \neq 0$. But $S(N) \subset S(I)$ and $i |_{S(I)}$ is injective. We deduce that $N = 0$. Since $\Lambda$ is connected graded and $S(I)^p$ has grading concentrated in a single degree, the injective module $\HOM_{k}(\Lambda,S(I)^p)$ has grading bounded above by $p$ and hence is locally nilpotent. Therefore, so too is $J$. This means that the short exact sequence 
    \[
    0 \to I \to J \to J/I \to 0
    \]
    has all terms lying in the category of locally nilpotent graded $\Lambda$-modules. Since $I$ is assumed to be injective in this subcategory, the inclusion $I \hookrightarrow J$ splits, i.e. $I$ is a graded direct summand of $J$. This implies that $I$ is injective as a graded $\Lambda$-module. 
\end{proof}

\begin{cor}\label{cor:equivalenceNoetherian}
   Assume that the underlying graded algebra $(A^!)^{\#}$ is graded left Noetherian. Every object $G(M)\in K(\Inj A^!)/\mc N$ is canonically isomorphic to $\mb G(M)$ in $K(\Inj A^!)_\nil$ and there is an equivalence 
    \[
\begin{tikzcd}
D(U) \ar[r, shift left=1.0ex, "\mb G"] & K(\Inj A^!)_\nil \ar[r,"\sim"] \ar[l, shift left=1.0ex, "\mb F"] & K(\mathrm{Inj}A^!)/\mathcal{N} .
\end{tikzcd}
\]
\end{cor}
\begin{proof}
    By Lemma \ref{lem:gradednoetherian}, $K(\Injn A^!)=K(\Inj A^!)_\nil$. The equivalence then follows from Theorem \ref{thm:equivalenceinjnilp}.
    Since $K(\Inj A^!)_\nil\cap \mc N=0$, we can view $K(\Inj A^!)_\nil$ as a subcategory of $K(\Inj A^!)/\mc N$. We have two functors $G,\mb G\colon K(U)\to K(\Inj A^!)$. For all $M\in K(U)$, $FG(M)$ is quasi-isomorphic to $F\mb G(M)$ by the proof of Theorem \ref{thm:equivalenceinjnilp}. Since $F$ is an equivalence, this implies that $G(M)$ and $\mb G(M)$ are isomorphic in $K(\Inj A^!)/\mc N$.
\end{proof}

\section{Applications}\label{sec:applications}

In this final section, we list immediate applications of the main theorem for non-homogeneous Koszul algebras. Combined with Theorem \ref{thm:equivalenceoffilteredcurved}, we also make statements about Koszul curved dg-algebras.

\subsection{$t$-structure}

Let $(\Lambda,d,c)$ be a \coconnective cdga. The equivalence of Corollary~\ref{cor:equivalenceNoetherian} endows the category $K(\Inj \Lambda)_{\mathrm{nil}}$ with a canonical $t$-structure provided that the underlying algebra $\Lambda^{\#}$ is Koszul and graded left Noetherian.  


\begin{prop}\label{prop:tstructure}
Let $(\Lambda,d,c)$ be a \coconnective cdga such that $\Lambda^{\#}$ is Koszul.
    The subcategories
 \begin{align*}
   K(\Injn \Lambda)^{\le 0} & = \{ I \in K(\Injn \Lambda) \, | \, H^{i}(S(I)) = 0, \, \forall \, i > 0 \} \\
   K(\Injn \Lambda)^{\ge 0} & = \{ I \in K(\Injn \Lambda) \, | \, H^{i}(S(I)) = 0, \, \forall \, i < 0 \} 
 \end{align*}
 are the aisle and coaisle respectively of a non-degenerate $t$-structure. 
\end{prop}

\begin{proof}
First, we note that Theorem~\ref{thm:equivalenceoffilteredcurved} says that there exists a non-homogeneous Koszul algebra $U$ such that $\Lambda = A^!$ as cdgas. Next, we note that $K(\Injn \Lambda)^{\le 0}$ is precisely the image of the standard aisle $D(U)^{\le 0} = \{ M \in D(U) \, | \, H^i(M) = 0, \, \forall \, i > 0 \}$. Indeed, writing $I = \mb G(M)$, the proof of Theorem~\ref{thm:equivalenceinjnilp} says that there is a quasi-isomorphism $S(I) \to F(\mb G(M)) \cong M$. Hence $H^i(M) = 0$ if and only if $H^i(S(I)) = 0$. Similarly,  $K(\Injn \Lambda)^{\ge 0}$ is the image of the standard coaisle $D(U)^{\ge 0} = \{ M \in D(U) \, | \, H^i(M) = 0, \, \forall \, i > 0 \}$. Since the standard $t$-structure on $D(U)$ is non-degenerate, it follows that the $t$-structure on $K(\Injn \Lambda)$ is also non-degenerate.
\end{proof}

If we reduce to complexes bounded above, we deduce similarly, by Corollary \ref{cor:main2abound}, that the subcategories 
 \begin{align*}
   K^-(\Inj \Lambda)^{\le 0} & = \{ I \in K^-(\Inj \Lambda) \, | \, H^{i}(S(I)) = 0, \, \forall \, i > 0 \} \\
   K^-(\Inj \Lambda)^{\ge 0} & = \{ I \in K^-(\Inj \Lambda) \, | \, H^{i}(S(I)) = 0, \, \forall \, i < 0 \} 
 \end{align*}
are the aisle and coaisle respectively of a (non-degenerate) $t$-structure, sent by $F$ to the standard $t$-structure on $D^-(U)$.  

If $\Lambda$ is also bounded then we deduce similarly, by Corollary~\ref{cor:main2b}, that the subcategories 
 \begin{align*}
   K(\Inj \Lambda)^{\le 0} & = \{ I \in K(\Inj \Lambda) \, | \, H^{i}(S(I)) = 0, \, \forall \, i > 0 \} \\
   K(\Inj \Lambda)^{\ge 0} & = \{ I \in K(\Inj \Lambda) \, | \, H^{i}(S(I)) = 0, \, \forall \, i < 0 \} 
 \end{align*}
are the aisle and coaisle respectively of a (non-degenerate) $t$-structure, sent by $F$ to the standard $t$-structure on $D(U)$.  



\begin{remark}
We expect that for any connected graded cdga, there is a $t$-structure on $K(\Injn \Lambda)$ whose aisle and coaisle are as in Proposition \ref{prop:tstructure} and that it is the one generated by the object corresponding to $k$ under the equivalence $K(\Injn \Lambda) \simeq D^{\mr{co}}({\Lambda})_\nil$ (where $k$ lives in the latter as it is a Verdier quotient of $K(\Lambda)_\nil$).
\end{remark}

\subsection{Bousfield localization}

Koszul duality, together with Bousfield localization, can be used to show that the inclusion $K(\Inj \Lambda)_\nil\hookrightarrow K(\Inj \Lambda)$ admits a left adjoint. 

\begin{prop}\label{prop:Bousfield}
Let $(\Lambda,d,c)$ be a \coconnective cdga such that $\Lambda^{\#}$ is Koszul and graded left Noetherian. There is a localization functor $L\colon K(\Inj \Lambda)\to K(\Inj \Lambda)$, such that $\Im L=K(\Inj \Lambda)_\nil$ and $\Ker L=\mc N$ where $\mc N$ is as in Subsection \ref{proofofMainThm}.
\end{prop}

\begin{proof}
Let $U$ be the non-homogeneous Koszul algebra associated to $(\Lambda,d,c)$ via Theorem~\ref{thm:equivalenceoffilteredcurved}. Let $\underline F$ be the composition of $F$ with the projection to the derived category: 
$$ K(\Inj \Lambda)\xrightarrow{F}K(U)\to D(U). $$
Let $\underline G$ be the composition $D(U)\rightarrow K(\Inj \Lambda)_\nil \hookrightarrow K(\Inj \Lambda)$, where the first functor is the equivalence from Corollary \ref{cor:equivalenceNoetherian}. In particular, $\underline G$ is fully faithful, so the composition $L = \underline G \circ \underline F$ is a localization functor by \cite[Proposition~2.4.1]{KrauseLocalization}. By Theorem \ref{thm:quasiisoJF}, $\Ker L=\mc N$. Since the restriction of $L$ to $K(\Inj \Lambda)_\nil$ is an auto-equivalence, it follows that $\Im L=K(\Inj \Lambda)_\nil$.
\end{proof}

By \cite[Propositions~4.9.1-4.10.1]{KrauseLocalization} we get as immediate consequence:
\begin{cor}
Let $(\Lambda,d,c)$ be a \coconnective cdga such that $\Lambda^{\#}$ is Koszul and graded left Noetherian. The inclusion $K(\Inj \Lambda)_\nil\hookrightarrow K(\Inj \Lambda)$ has a left adjoint, and 
    $K(\Inj \Lambda)_\nil \cong K(\Inj \Lambda) / \mc{N}$.
\end{cor}

Similarly, when $\Lambda$ is finite-dimensional over $\Bbbk$, we can use Koszul duality to construct a left adjoint to the inclusion $K(\Inj \Lambda) \hookrightarrow K(\Lambda)$. First, let $\mc{K}$ be the full subcategory of $K(\Lambda)$ consisting of all objects $N$ such that $F(N)$ is an acyclic $U$-module. 

\begin{prop}
Let $(\Lambda,d,c)$ be a $\Bbbk$-finite-dimensional \coconnective cdga such that $\Lambda^{\#}$ is Koszul. There is a localization functor $L\colon K(\Lambda)\to K(\Lambda)$, such that $\Im L=K(\Inj \Lambda)$, $\Ker L=\mc K$.
\end{prop}

The proof is identical to that of Proposition~\ref{prop:Bousfield}. We get an immediate consequence:

\begin{cor}
Let $(\Lambda,d,c)$ be a $\Bbbk$-finite-dimensional \coconnective cdga such that $\Lambda^{\#}$ is Koszul. The inclusion $K(\Inj \Lambda) \hookrightarrow K(\Lambda)$ has a left adjoint, and 
    $K(\Inj \Lambda) \cong K(\Lambda) / \mc{K}$.
\end{cor}

\subsection{K-theory}

In this section we note that Koszul duality can be used to compute the $K$-theory of some cdgas. Recall that for any ring $R$, there are groups $K_i(R)$ for $i \geq 0$ such that $K_0(R)$ is the Grothendieck group of finitely generated projective modules. $K$-theory can be defined for enhanced triangulated categories and one has $K_i(R) \simeq K_i(D^{\mr{perf}}(R))$ for all $i \geq 0$, where $D^{\mr{perf}}(R) \subseteq D(R)$ consists of the perfect complexes. See \cite[(V.2.7.2)]{Kbook} for more details. As the perfect complexes are the compact objects in ${D}(R)$, it is reasonable to define the $K$-theory of a cdga as the $K$-theory of the compact objects in its coderived category. 

Recall that the coderived category $D^{\mr{co}}({\Lambda})_\nil$ of a nonnegatively graded cdga $\Lambda$ is defined in \cite[Definition~6.11]{positselskiRelativeNonhomogeneousKoszul2021} (where the notation $D^{\mr{co}}(\mathsf{comod}\textrm{-}\Lambda)$ is used). If each $\Lambda^i$ is finite-dimensional, then the coderived category is equivalent to $K(\Injn \Lambda)$ by \cite[Theorem 8.17]{positselskiRelativeNonhomogeneousKoszul2021}. Note that $D^{\mr{co}}({\Lambda})_\nil$ (as well as all of its full subcategories) admit a dg-enhancement. Indeed, it is the Verdier localisation of a full subcategory of $K(\Lambda)$ which admits a DG-enhancement by Remark \ref{rem:dgenhancement}.  The existence of the enhancement of $D^{\mr{co}}({\Lambda})_\nil$ then follows from Drinfeld's quotient construction \cite{DRINFELD2004643}. Let $D^{\mr{co}}({\Lambda})_\nil^c$ denote the subcategory of compact objects in the triangulated category $D^{\mr{co}}({\Lambda})_\nil$. As a subcategory, it too admits a natural DG-enhancement and so by taking its dg-nerve, and applying Example 2.11 and Definition 10.1 in \cite{barwickalgktheory}, we can define its $K$-theory. For $i \geq 0$, we set
\[
K_i(\Lambda) \coloneq K_i(D^{\mr{co}}({\Lambda})_\nil^c) 
\]

For those interested only in Grothendieck groups, the discussion of enhancements can be ignored as $K_0(\Lambda)$ is just the Grothendieck group of the triangulated category $D^{\mr{co}}({\Lambda})_\nil^c$

\begin{prop}
Suppose $(\Lambda,d,c)$ is a cdga over $k$ such that $\Lambda^\#$ is Koszul, bounded and the quadratic dual of $\Lambda^\#$ is left Noetherian. Then $K_i(\Lambda) \cong K_i(k)$ for all $i \geq 0$.   
\end{prop}

\begin{proof}
By Theorem \ref{thm:equivalenceoffilteredcurved}, there is a non-homogeneous Koszul algebra $U$ whose Koszul dual is $\Lambda$. The triangle equivalence ${D}(U) \simeq D^{\mr{co}}({\Lambda})_\nil$ in Corollary 6.18 of \cite{positselskiRelativeNonhomogeneousKoszul2021} restricts to a triangle equivalence ${D}^{\mr{perf}}(U) \simeq {D}^{\mr{co}}(\Lambda)^c$ between the compact objects. By defintion of the inverse functors, it is clear this equivalence comes from a quasi-equivalence between the dg-enhancements. Hence, there are isomorphisms $K_i(U) \cong K_i({D}^{\mr{perf}}(U)) \cong K_i(D^{\mr{co}}({\Lambda})^c_\nil) = K_i(\Lambda)$ for all $i$. By definition $U$ is filtered and its associated graded ring $A$ is the quadratic dual of $\Lambda^\#$. By Lemma \ref{lem:finglbdim}, $A$ has finite global dimension and by assumption it is left Noetherian. By Corollary 6.18 of \cite{MR}, $U$ also has finite global dimension. By Remark 6.4.1 in \cite{Kbook}, we see that $K(U) \simeq K(k)$. 
\end{proof}

\subsection{Free resolutions}

We now consider applications of the equivalence to the representation theory of non-homogeneous Koszul algebras $U$. As in the graded case, filtered Koszul duality gives rise to an explicit projective resolution of any given $U$-module. Let $U$ be a non-homogeneous Koszul algebra. The quasi-isomorphism $FG \iso \mr{Id}$ applied to $U$ gives a resolution of $U$-$U$-bimodules
\begin{equation}\label{eq:projUresolution}
      FG(U) = U \otimes_{\rin} {}^* (A^!) \otimes_{\rin} U \to U.
\end{equation}
Notice that the complex $U \otimes_{\rin} {}^* (A^!) \otimes_{\rin} U$ is finite length if and only if $A^!$ is finite-dimensional, though this does not necessarily imply that $U$ has infinite global dimension if $A^!$ is infinite-dimensional. 

\begin{remark}\label{rem:freeresolution}
In the case of the deformed preprojective algebra associated to a finite connected non-Dynkin quiver, the resolution \eqref{eq:projUresolution} recovers the resolution constructed by Crawley-Boevey; see \cite[Theorem~2.7]{crawley-boeveyDeformedPreprojectiveAlgebras2022}.  
\end{remark}

If $M$ is any left $U$-module then tensoring \eqref{eq:projUresolution} on the right by $M$ gives rise to an explicit resolution
\begin{equation}\label{eq:projUresolution2}
  FG(M) = U \otimes_{\rin} {}^*(A^!) \otimes_{\rin} M \to M.
\end{equation}
If $M,N$ are left $U$-modules then using resolution \eqref{eq:projUresolution2} we see that there exists a differential (squaring to zero) on $\HOM_{\rin}(M,{}^! A \otimes_{\rin} N)$ such that 
\begin{align*}
\Ext^{i}_U(M,N) & \cong H^i(\underline{\Hom}_{U}(U \otimes_{\rin} {}^* (A^!) \otimes_{\rin} M,N)) \\
& \cong H^i(\underline{\Hom}_{{\rin}}({}^* (A^!) \otimes_{\rin} M,N)) \\ 
& \cong H^{i}(\underline{\Hom}_{{\rin}}(M,{}^! A \otimes_{\rin} N)),
\end{align*}
where we think of $M,N$ as complexes concentrated in one degree. 

We can apply the same argument to compute the Hochschild cohomology of $U$. Namely, 
\begin{align*}
HH^i(U) = \mathrm{Ext}^{i}_{U^e}(U,U) & \cong H^{i}({}^!A \otimes_{\rin} U).
\end{align*}
This recovers a result of Negron \cite[\S~8]{negronCupProductHochschild2017}. 

Recall that the algebra $U$ is a quotient of the tensor algebra $T_{\rin} E$. An irreducible $U$-module $\lambda$ is called \textit{rigid} if $E \cdot \lambda = 0$. Though rigid $A$-modules always exist, being the inflation of irreducible ${\rin}$-modules to $A$, they need not exist for $U$. In the case of symplectic reflection algebras, it is a non-trivial problem to classify rigid modules \cite{BellThielCusp,CiubotaruOne,MontaraniEtingof}. If $\lambda$ is rigid then one can check that the differential on $\underline{\Hom}_{{\rin}}(\lambda,{}^! A \otimes_{\rin} \lambda)$ vanishes and hence:

\begin{prop}\label{prop:rigidmodules}
    If $\lambda$ is a rigid $U$-module then $\mathrm{Ext}^{\idot}_U(\lambda,\lambda) = \underline{\Hom}_{{\rin}}(\lambda,{}^! A \otimes_{\rin} \lambda)$. 
\end{prop}

Here the product in $\underline{\Hom}_{{\rin}}(\lambda,{}^! A \otimes_{\rin} \lambda)$ is given by $\phi_1 \cdot \phi_2 = (m_{A^!} \ot \id) \circ (\id \ot \phi_1) \circ \phi_2$. 

\subsection{Proof of Proposition~\ref{prop:koszuldualext}}\label{sec:PropExtproof} 

We need to show that if $A$ is a Koszul $k$-algebra then there are isomorphisms of graded $k$-algebras 
     $$
     ({}^! A)^{\op} \cong \Ext_{A}^{\idot}({}_A k,{}_A k), \quad A^! \cong \Ext_{A}^{\idot}(\rin_A ,\rin_A).
     $$   
We define $\psi \colon {}^! A \to \underline{\Hom}_{{\rin}}( \rin,{}^! A \otimes_{\rin} \rin)$ by $\psi(b)(r) = rb \ot 1$. Then 
\[
\psi(b_1 b_2)(r) = r b_1 b_2 \ot 1 = (\psi(b_2) \cdot \psi(b_1))(r)
\]
which shows that $\psi$ is a ring isomorphism $ ({}^! A)^{\op}  \to \underline{\Hom}_{{\rin}}(\rin ,{}^! A \otimes_{\rin} \rin)$. It follows from Proposition~\ref{prop:rigidmodules} that $({}^! A)^{\op} \cong \Ext_{A}^{\idot}({}_A k,{}_A k)$. The second isomorphism is similar. 

\appendix

\section{Monoidal Pairs}\label{appendix}

\subsection{Monoidal pairs}

In this section let $(\mathcal{M},\otimes,\rin)$ denote an abelian category with an exact closed rigid monoidal structure. For instance, $\mathcal{M}$ is the category of finitely generated bimodules over a finite-dimensional semisimple $\Bbbk$-algebra $\rin$. Closed means that the monoidal category admits an internal hom functor $\hom(-,-)$. Let $(-)^\ast := \hom(-,\rin): \mathcal{M} \to \mathcal{M}^{op}$. Then $\mathcal{M}$ rigid means that the dual $M^\ast$ of any object $M \in \mathcal{M}$ satisfies a certain list of axioms, as given for instance in \cite[Section~2.10]{etingofTensorCategories2015}. In particular, for any object $M$ there are adjunctions $- \otimes M \dashv - \otimes M^\ast$ and $M^\ast - \otimes  \dashv M \otimes -$. The unit and counit of these adjunctions are given by tensoring with so-called evaluation and coevaluation maps $e_M \colon M^\ast \otimes M \to \rin$ and $c_M \colon \rin \to M \otimes M^\ast$. Furthermore, the following composition is the identity
\[
M \xrightarrow{c_M \otimes M} M \otimes M^\ast \otimes M \xrightarrow{M \otimes \, e_M} M.
\]
It also follows that $(-)^\ast$ is a strong anti-monoidal functor, i.e., there are natural isomorphisms 
\[
\tau_{M,N}:  (M \otimes N)^\ast \simeq N^\ast \otimes M^\ast.
\]
Recall that the dual of a map $f:M\rightarrow N$ can be defined in terms of evaluation and coevaluation \cite[(2.47)]{etingofTensorCategories2015}:
\begin{equation}\label{eq:defdual}
    f^*:N^*\xrightarrow{1\ot \, c_M}N^*\ok M\ok M^*\xrightarrow{1\ot f\ot 1}N^*\ok N\ok M^*\xrightarrow{e_N\ot 1}M^*.
\end{equation}
We will use the following facts.

\begin{prop}\label{prop:monoidalnonsense}
The following diagrams commute for any $M,N \in \mathcal{M}$ and morphism $f \colon M \to N$.
\[
\begin{tikzcd}
\rin \arrow[r,"c_M"] \arrow[d,"c_{M \otimes M}"] & M \otimes M^\ast \arrow[d,"M \otimes \, c_M \otimes M^\ast"] \\  
 (M \otimes M) \otimes (M \otimes M)^\ast \arrow[r,"\sim","1 \otimes \, \tau_{M,M}"'] & M \otimes M \otimes M^\ast \otimes M^\ast 
\end{tikzcd}
\]
\[
\begin{tikzcd}
N^\ast \otimes M \arrow[d,"N^\ast \otimes f"] \arrow[r,"f^\ast \otimes M"] & M^\ast \otimes M \arrow[d,"e_M"] \\
N^\ast \otimes N \arrow[r,"e_N"] & \rin \\
\end{tikzcd}
\]
\end{prop}

\begin{proof} 
The commutativity of the first diagram follows from the uniqueness of the dual \cite[Proposition 2.10.5]{etingofTensorCategories2015}. See Exercise 2.10.7 (b) from \cite{etingofTensorCategories2015}.

Notice that we have two adjunctions that give bijections
$$\Hom(N^*\ot M,\rin)\cong \Hom(M,N), $$
$$\Hom(N^*\ot M,\rin)\cong\Hom(N^*,M^*). $$
The bottom left half of the diagram is the image of $f\in\Hom(M,N)$ under the first adjunction, while the top right half is the image of $f^*\in\Hom(N^*,M^*)$ under the second. It is thus sufficient to show that the following diagram commutes:
\begin{equation*}
\begin{tikzcd}
    & \Hom(N^*\ot M) &  \\
{\Hom(M,N)} \arrow[rr, "(-)^*"] \arrow[ru] &                & {\Hom(N^*,M^*)} \arrow[lu]
\end{tikzcd}
\end{equation*}
We can check directly that the composition $\Hom(M,N)\rightarrow\Hom(N^*\ot M)\rightarrow\Hom(N^*,M^*)$ sends $f$ to the composition \eqref{eq:defdual}, which by definition is equal to $f^*$.
\end{proof}

\begin{lem}\label{lemma:adjunctioncoev}
The adjunct of a composite
\[
P \xrightarrow{i} F \xrightarrow{p} B
\]
across $- \otimes P \dashv - \otimes P^\ast$ is 
\begin{equation}\label{eq:latteradjunction}
    \rin \xrightarrow{c_F} F \otimes F^\ast \xrightarrow{p \otimes i^\ast} B \otimes P^\ast.
\end{equation}
\end{lem}

\begin{proof}
By definition, the adjunct of \eqref{eq:latteradjunction} is
\[
P \xrightarrow{c_F \otimes P} F \otimes F^\ast \otimes P \xrightarrow{p \otimes i^\ast \otimes P} B \otimes P^\ast \otimes P \xrightarrow{B \otimes e_P} B. 
\]
This is the same as the top row in the diagram below
\[
\begin{tikzcd}
P \arrow[d,"i"'] \arrow[r,"c_F \otimes P"] & F \otimes F^\ast \otimes P \arrow[rr,"F \otimes i^\ast \otimes P"] \arrow[d,"F \otimes F^\ast \otimes i"'] & & F \otimes P^\ast \otimes P \arrow[r,"F \otimes e_P"]  \arrow[d,"F \otimes e_P"'] & F \arrow[dl,equal] \arrow[r,"p"]&  B \\
F  \arrow[r,"c_F \otimes F"] & F \otimes F^\ast \otimes F \arrow[rr," F \otimes e_F"] & & F & & 
\end{tikzcd}
\]
The left square clearly commutes, the middle square is $F$ tensored with the second diagram in Proposition~\ref{prop:monoidalnonsense}. The bottom horizontal composition is the identity, and so we are done. 
\end{proof}

\printbibliography

\end{document}